%% file: product-group.tex
\documentclass[11pt]{article}

\input{preamble.tex}

\usepackage{fullpage}
\usepackage[nomarkers,figuresonly,nofiglist]{endfloat}

\usepackage[backend=bibtex]{biblatex}
\begin{document}

\title{Type AIII orbits in the affine flag variety of type A}

\author{
Kam Hung TONG\thanks{Department of Applied Mathematics, The Hong Kong Polytechnic University \\
\hspace*{1.8em}E-mail address: {\tt kam-hung-terry.tong@polyu.edu.hk}.
}
}
\date{\today}

\maketitle

\begin{abstract}
Matsuki and Oshima introduced the notion of clans, which are incomplete matchings with positive or negative signs on isolated vertices. They discovered that clans parametrise $K$-orbits in the flag varieties for classical linear groups, where $K$ is a fixed point subgroup of an involution in the same classical linear group. In this work we investigate the affine version of these orbits. For a field $\K$ with characteristic not equal to two, we construct an explicit bijection between the $\GL_p(\KLt) \times \GL_q(\KLt)$-orbits in the affine flag variety and certain objects called affine $(p,q)$-clans. These affine $(p,q)$-clans can be concretely interpreted as involutions in the affine permutation group with positive or negative signs on fixed points.
\end{abstract}

\setcounter{tocdepth}{2}
\tableofcontents

\section{Introduction}
This paper is the continuation of \cite{tong2024orthogonalsymplecticorbitsaffine}. Let $G$ be a connected reductive algebraic group over the field of complex numbers $\CC$, $B$ a Borel subgroup of $G$, and $K$ the fixed point subgroup of a holomorphic involution of $G$. In this classical setting, the $K$-orbits on the \defn{flag variety} $G/B$ are well-studied. Matsuki and Oshima \cite{MO} introduced the notion of \defn{clans}, which are incomplete matchings with positive or negative signs on isolated vertices. They discovered that clans parametrise $K$-orbits in the flag varieties for classical linear groups. Yamamoto \cite{Yamamoto} gave a full proof of this classification for the $\GL_p(\CC) \times \GL_q(\CC)$-orbits in the type A flag variety, which are often called type AIII orbits.

Can, Joyce, and Wyser \cite{CJW} studied the maximal chains in the weak order poset for the three types of $K$-orbits when $G = \GL_n(\CC)$, including the type AIII orbits, and specifically described a formula for Schubert classes. Woo and  Wyser \cite{woo2015combinatorial}, and Burks and Pawlowski \cite{burks2018reduced} studied notions of pattern avoidance and reduced words for clans and related them to the closures of the $K$-orbits in the flag variety.
More applications can be found in \cite{colarusso2014gelfand, mcgovern2009closures, mcgovern2009pattern, wyser2014polynomials, wyser2017polynomials}. 
 
This work classifies the affine version of the type AIII orbits. To be more specific, suppose that $\KLt$ is the field of formal Laurent series over a field $\K$ with characteristic not equal to two, and $\KPt$ is the ring of formal power series over $\K$. Let $G = \GL_n(\KLt)$ be the general linear group over $\KLt$, $B$ the Iwahori subgroup of $G$ consisting of all upper-triangular matrices modulo $t$ in $\GL_n(\KPt)$, and $K = \GL_p(\KLt) \times \GL_q(\KLt)$ with $n = p+q$. In this work, we study the orbits of $K$ on the \defn{affine flag variety} $\GL_n(\KLt)/B$. 
\subsection{The classical background}
We first briefly review the results in the classical case.

Denote $I_n$ to be the $n$-by-$n$ identity matrix. Let $G=\GL_n(\CC)$, let $B \subset G$ be the Borel subgroup of $G$ consisting of upper triangular matrices, and let $\theta$ be the involution of $G$ defined by 
\[\theta(g) = \begin{pmatrix} I_p & 0 \\ 0 & -I_q \end{pmatrix} g \begin{pmatrix} I_p & 0 \\ 0 & -I_q \end{pmatrix}\]
 for integers $p, q \geq 0$ with $n = p+q$. The fixed point subgroup $K = G^{\theta} := \{g \in G : \theta(g) = g\}$ is the product group 
 \[\GL_p(\CC) \times \GL_q(\CC) = \left\{\begin{pmatrix} k_{1} & 0 \\ 0 & k_{2} \end{pmatrix} : k_{1} \in \GL_p(\CC), k_{2} \in \GL_q(\CC)\right\}.\]
 In the classification of involutions for $\GL_n(\CC)$, this $\theta$ is called  type AIII. On the other hand, the involutions $\theta$ where $K = G^{\theta}= \O_n(\CC)$ or $\Sp_n(\CC)$ are called type AI and type AII, respectively. See \cite[Chapter 5, \S1.5]{onishchik2012lie} for more details.


Denote $\ZZ$ to be the ring of integers. For integers $a,b$, denote $[a,b] = \{m \in \ZZ: a \leq m \leq b\}$ and denote $[a] = [1,a]$. A \defn{(complete) flag} is a strictly increasing sequence of vector subspaces 
\[\{0\} = V_0 \subset V_1 \subset V_2 \subset \dots \subset V_n = \CC^n,\]
 and there is an explicit bijection between $\GL_n(\CC)/B$ and the set $Fl_n$ consisting of all complete flags. The bijection is as follows: for any $gB \in \GL_n(\CC)/B$, denote $g_i$ to be the $i$-th column of $g$ for $i \in [1, n]$, define $V_0 = \{ 0 \}$ and define $V_i = \CC g_1 \oplus \CC g_2 \oplus \dots \oplus \CC g_i$ for $i \in [n]$. Then $\{0\} = V_0 \subset V_1 \subset V_2 \subset \dots \subset V_n = \CC^n$ is a complete flag.

A \defn{$(p,q)$-clan} is an $n$-tuple $(c_1, c_2, \dots c_n)$ such that each $c_i$ is either $+$, $-$ or a natural number, such that every natural number, if it appears, appears exactly twice, and the number of $+$ signs minus the number of $-$ signs must be $p-q$. 
Two such $n$-tuples $c = (c_1, c_2, \dots c_n)$ and $d = (d_1, d_2, \dots d_n)$ are equivalent if the following holds:
\ben
\item[(i)] For all $i \in [n]$, $c_i = +$ if and only if $d_i = +$, and $c_i = -$ if and only if $d_i = -$.
\item[(ii)] For all $i,j \in [n]$ with $i \neq j$, $c_i = c_j \in \NN$ if and only if $d_i = d_j \in \NN$.
\een
We do not distinguish between $(p,q)$-clans that are equivalent in the above sense. As a remark, a $(p,q)$-clan contains the same data as an involution in the  permutation group $S_n$ with fixed points labeled by $+$ or $-$.

The $(p,q)$-clans can be represented by arc diagrams, which contain $n$ points lying in a row labelled as $1, 2, \dots, n$, with an arc joining points $i$ and $j$ if $c_i = c_j \in \NN$, and plus or minus signs labelled on the points $k$ if $c_k = +$ or $-$ respectively. The following are the arc diagrams for the $(4,3)$-clans $(1, 2, +, + , - , 2, 1)$ and $(+, 1, 2, 1, 3, 3, 2)$ respectively:
\begin{center}
\begin{tikzpicture}[xscale=0.5, yscale=0.35,>=latex]
\node at (-3,0) (1) {$\circ$};
\node at (-2,0) (2) {$\circ$};
\node at (-1,0) (3) {$\circ$};
\node at (0,0) (4) {$\circ$};
\node at (1,0) (5) {$\circ$};
\node at (2,0) (6) {$\circ$};
\node at (3,0) (7) {$\circ$};
\node at (0,-0.5) (8) {$+$};
\node at (-1,-0.5) (9) {$+$};
\node at (1,-0.5) (10) {$-$};

  \draw[-, thick] (1.north) to [out=75, in = 105] (7.north) ;
    \draw[-, thick] (2.north) to [out=75, in = 105] (6.north) ;
\end{tikzpicture}
\quad
\begin{tikzpicture}[xscale=0.5, yscale=0.35,>=latex]
\node at (-3,0) (1) {$\circ$};
\node at (-2,0) (2) {$\circ$};
\node at (-1,0) (3) {$\circ$};
\node at (0,0) (4) {$\circ$};
\node at (1,0) (5) {$\circ$};
\node at (2,0) (6) {$\circ$};
\node at (3,0) (7) {$\circ$};
\node at (-3,-0.5) (8) {$+$};

  \draw[-, thick] (2.north) to [out=75, in = 105] (4.north) ;
    \draw[-, thick] (3.north) to  [out=75, in = 105] (7.north) ;
     \draw[-, thick] (5.north) to  [out=75, in = 105] (6.north) ;
     \end{tikzpicture}.
\end{center}

In \cite[Proposition 2.2.6]{Yamamoto}, Yamamoto gave an explicit inductive algorithm to produce a $(p,q)$-clan $c(x) = (c_1, c_2, \dots,  c_n)$ from a flag $x = (V_0, V_1, \dots, V_n)$.
The algorithm involves some dimensional $K$-invariants of $x$; see~\cite[\S2.2]{Yamamoto}. 
Yamamoto proved that this algorithm is a bijection between the $K$-orbits on the set $Fl_n$ and $(p,q)$-clans \cite[Theorem 2.2.8]{Yamamoto}. Our goal here is to describe an affine generalization of the $K$-orbits.

%

\subsection{The affine type AIII orbits in affine flag variety}

In this paper we study the affine analog of type AIII $K$-orbits in the flag variety in the following sense. Let $\KLt$ be the field of formal Laurent series in $t$ consisting of all the formal sums $ \sum^{\infty}_{i\geq N} a_i t^i$, in which $N \in \ZZ$ and $a_i \in \K$ for $i \geq N$. Let $\KPt$ be the ring of formal power series consisting of all the formal sums $ \sum^{\infty}_{i\geq 0} a_i t^i$, in which $a_i \in \K$ for $i \geq 0$. Let $\KPt^*$ be the multiplicative group consisting of invertible elements in $A$. Equivalently, $A^*$ consists of all elements in $A$ with non-zero constant terms. Define the order \defn{$\ord$} of a formal Laurent series $f(t) =  \sum^{\infty}_{i = N} a_it^i \in \KLt$ to be the smallest integer $N$ such that $a_N\neq 0$. 

We redefine $G = \GL_n(\KLt)$ to be the group of invertible $n$-by-$n$ matrices over $\KLt$ and redefine $B$ to be the subgroup consisting of all upper triangular modulo $t$ matrices in $\GL_n(\KPt)$, that is, invertible matrices with entries in $\KPt$ that become upper triangular if we set $t = 0$ for these matrices. The subgroup $B$ is often called the \defn{Iwahori subgroup}, and the set of cosets $G/B$ is often called the \defn{affine flag variety}. The affine flag variety has been studied in \cite{lam2021back, lee2019combinatorial, Nadler04}, for example.

We study the $K:= \GL_p(\KLt) \times \GL_q(\KLt)$-orbits in the affine flag variety by using an analogue of Yamamoto's work in the classical case.
\begin{definition}\label{affine-clan-def}
An \defn{affine $(p,q)$-clan} is a $\ZZ$-indexed sequence $c = (\dots,c_1,c_2,c_3,\dots)$ with $n = p+q$
such that 
\ben
\item[(i)]each $c_i$ is either $+$, $-$ or an integer;
\item[(ii)]for $k \in \ZZ$, $c_{i+kn} = c_{i}+kn$ if $c_i$ is an integer and $c_{i+n} = c_i$ if $c_i$ is $+$ or $-$;
\item[(iii)] $\#\{ i \in [n] : c_i = +\} - \#\{ i \in [n] : c_i = -\} = p-q$,
and every integer, if it appears, appears exactly twice in the sequence.
\een

Two such $\ZZ$-indexed sequences $c = (\dots,c_1,c_2,c_3,\dots)$ and $d = (\dots,d_1,d_2,d_3,\dots)$ are equivalent if the following holds:
\ben
\item[(i)] For all $i \in \ZZ$, $c_i = +$ if and only if $d_i = +$, and $c_i = -$ if and only if $d_i = -$.
\item[(ii)] For all $i,j \in \ZZ$ with $i \neq j$, $c_i = c_j \in \ZZ$ if and only if $d_i = d_j \in \ZZ$.
\een
We do not distinguish between affine $(p,q)$-clans that are equivalent in the above sense. 
\end{definition}

Note that from the periodicity in the definition of affine $(p,q)$-clans above, the $n$ elements $c_1, c_2, c_3, \dots, c_n$ in any affine $(p,q)$-clan $c = (\dots,c_1,c_2,c_3,\dots)$ determine the whole $\ZZ$-indexed sequence $c$. Therefore we can denote an affine $(p,q)$-clan $c$ as $c = (c_1, c_2, \dots, c_n)$.
\begin{example}
The affine $(1,1)$-clans are $(+,-)$, $(-,+)$ and $(1, 1+2k)$ for $k \in \ZZ$. The affine $(2,1)$-clans are $(+, + ,-)$, $(+, -, +)$, $(-, +, +)$, $(1, 1+3k, +)$, $(+, 1, 1+3k)$ and $(1, +, 1+3k)$ for $k \in \ZZ$.
\end{example}
A useful graphical method of representing the affine $(p,q)$-clans is through the \defn{winding diagrams} of affine $(p,q)$-clans:
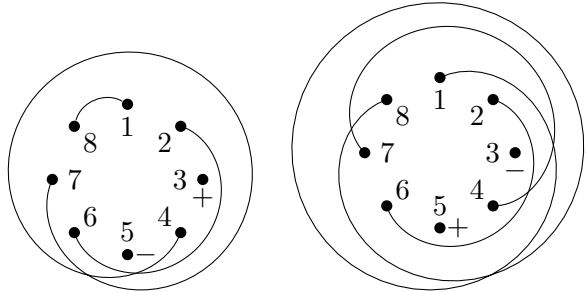
\begin{center}
\begin{tikzpicture}
  \foreach \x in {1,...,8}
  {
    \node at ({90-(360/8)*(\x-1)}:1) [circle,fill,inner sep=1.5pt] (p\x) {};
    \node at ({90-(360/8)*(\x-1)}:0.7) [draw = none, fill=none] (q\x) {\x};
  }
\node at (1,-0.24) [draw = none, fill=none] (plus) {$+$};
\node at (0.24,-1) [draw = none, fill=none] (minus) {$-$};
\draw (p8) to[out=75, in=150] (p1);
\draw [-,>=latex,domain=0:100,samples=100] plot ({(1.0 + 0.3 * sin(180 * (0.5 + asin(-0.9 + 1.8 * (\x / 100)) / asin(0.9) / 2))) * cos(90 - (45 + \x * 1.8))}, {(1.0 + 0.3 * sin(180 * (0.5 + asin(-0.9 + 1.8 * (\x / 100)) / asin(0.9) / 2))) * sin(90 - (45 + \x * 1.8))});
\draw [-,>=latex,domain=0:100,samples=100] plot ({(1.0 + 0.7 * sin(180 * (0.5 + asin(-0.9 + 1.8 * (\x / 100)) / asin(0.9) / 2))) * cos(-45 - (0.0 + \x * 4.95))}, {(1.0 + 0.7 * sin(180 * (0.5 + asin(-0.9 + 1.8 * (\x / 100)) / asin(0.9) / 2))) * sin(-45 - (0.0 + \x * 4.95))});
\end{tikzpicture}
\quad
\begin{tikzpicture}
  \foreach \x in {1,...,8}
  {
    \node at ({90-(360/8)*(\x-1)}:1) [circle,fill,inner sep=1.5pt] (p\x) {};
    \node at ({90-(360/8)*(\x-1)}:0.7) [draw = none, fill=none] (q\x) {\x};
  }
\node at (1,-0.24) [draw = none, fill=none] (minus) {$-$};
\node at (0.24,-1) [draw = none, fill=none] (plus) {$+$};
\draw [-,>=latex,domain=0:100,samples=100] plot ({(1.0 + 0.3 * sin(180 * (0.5 + asin(-0.9 + 1.8 * (\x / 100)) / asin(0.9) / 2))) * cos(90 - (45 + \x * 1.8))}, {(1.0 + 0.3 * sin(180 * (0.5 + asin(-0.9 + 1.8 * (\x / 100)) / asin(0.9) / 2))) * sin(90 - (45 + \x * 1.8))});
\draw [-,>=latex,domain=0:100,samples=100] plot ({(1.0 + 0.7 * sin(180 * (0.5 + asin(-0.9 + 1.8 * (\x / 100)) / asin(0.9) / 2))) * cos(-45 + (0.0 + \x * 2.25))}, {(1.0 + 0.7 * sin(180 * (0.5 + asin(-0.9 + 1.8 * (\x / 100)) / asin(0.9) / 2))) * sin(-45 + (0.0 + \x * 2.25))});
\draw [-,>=latex,domain=0:100,samples=100] plot ({(1.0 + 1 * sin(180 * (0.5 + asin(-0.9 + 1.8 * (\x / 100)) / asin(0.9) / 2))) * cos(90 - (0.0 + \x * 6.75))}, {(1.0 + 1 * sin(180 * (0.5 + asin(-0.9 + 1.8 * (\x / 100)) / asin(0.9) / 2))) * sin(90 - (0.0 + \x * 6.75))});
\end{tikzpicture}
\end{center}
Here the numbers $1, 2, \dots n$ are arranged in order around a circle. In such diagrams, a curve travelling $k$ times clockwise connecting $i <j$ means that $c_j = c_i - kn$, and a curve travelling $k-1$ times anticlockwise connecting $i<j$ means that $c_j = c_i + kn$. The examples above are the affine $(4,4)$-clans with $(c_1, c_2, c_3, c_4, c_5, c_6, c_7, c_8) = (1, 2, +, 3, -, 2, 2-8, 1+8)$ and $(1, 2, -, 3, +, 2, 2+8, 1-8)$ respectively. One can think of affine $(p,q)$-clans as involutions in the \defn{affine symmetric group} with $+$ or $-$ signs assigned to the fixed points.


Our main theorem is the following.

\begin{theorem}\label{main-thm-lattice}
There is a natural bijection between $K$-orbits in $G/B$ to affine $(p,q)$-clans.
\end{theorem}
There is a natural bijection between the two common notions of affine flag variety: $G/B$ and the space $\widetilde{Fl}_n$, consisting of all \defn{affine flags} of lattices $\Lambda_{\bullet} = (\Lambda_1 \supset \Lambda_2 \supset \dots \supset \Lambda_n)$ such that all $\Lambda_i$'s are full-rank $\KPt$-submodules in $\KLt^n$, $\Lambda_n \supset t \Lambda_1$ and $\dim_{\K}(\Lambda_j / \Lambda_{j+1}) = 1$ for $j \in [1,n-1]$. See Section~\ref{affine-flag-variety}.
In this work, we prove the equivalent statement that there is an explicit bijection between the orbits of $K = \GL_p(\KLt) \times \GL_q(\KLt)$ in the affine flag variety $\widetilde{Fl}_n$ and affine $(p,q)$-clans, by considering the properties of left multiplication of elements in $K$ on the affine flags.

%
An explicit map from an affine flag $\Lambda_\bullet$ to an affine $(p,q)$-clan $c(\Lambda_{\bullet})$ is given in Definition~\ref{alg-affine-flag-to-affine-clan}. By Proposition~\ref{affine-flag-to-affine-clan}, affine flags in the same $K$-orbit are mapped to the same affine $(p,q)$-clan. Therefore, the map in Definition~\ref{alg-affine-flag-to-affine-clan} induces a map from the $K$-orbits in $\widetilde{Fl}_n$ to the set of affine $(p,q)$-clans.
In the proof of Proposition~\ref{affine-clan-to-affine-flag}, an explicit inverse map from the set of affine $(p,q)$-clans to the $K$-orbits in $\widetilde{Fl}_n$ is given.

Presenting distinct orbit representatives in terms of affine $(p,q)$-clans is combinatorial in nature, hence accessible for future work. For instance, a possible weak order on affine $(p,q)$-clans can be readily developed, similar to the classical case in \cite{CJW, Richardson90}. Another possible application is that the distinct orbit representatives could provide explicit examples of \defn{affine Lusztig-Vogan modules} and positivity of the \defn{affine Kazhdan-Lusztig-Vogan polynomials} defined in~\cite{CY}. Studying the orbit closures also has applications to relative Langlands duality, as described in~\cite{CY}.



\section*{Acknowledgment}
We thank Peter Magyar, Eric Marberg, and Yongchang Zhu for many useful comments and discussions.

\section{Preliminaries}\label{chpt2-ref}

In this section, we introduce the necessary notations and preliminaries for this paper, focusing on the concept of loop groups and the affine flag variety. These topics have various approaches, but Magyar's work~\cite{Magyar} provides a comprehensive exposition on them. For more detailed information, refer to~\cite{Kumar, PS}.
\subsection{Loop groups}
As before, let $\K$ be a field of characteristic not equal to two.
The notions of the field of \defn{formal Laurent series} $F = \KLt$, the ring of \defn{formal power series} $A = \KPt$, the \defn{(algebraic) loop group} of $\GL_n(\K)$, namely $G = \GL_n(F)$, and the \defn{Iwahori subgroup} $B \subset \GL_n(A)$ consisting of upper triangular matrices modulo $t$, are the same as in the introduction.

In this paper let 
\begin{align*} 
K &= \left\{ \begin{pmatrix} k_1 & 0 \\ 0 & k_2 \end{pmatrix} : k_1 \in \GL_p(F), k_2 \in \GL_q( F) \right\} \\
& \cong \GL_p(F) \times \GL_q(F).
\end{align*}
This group is the \defn{(algebraic) loop group} of $\GL_p(\K) \times \GL_q(\K)$.

\subsection{Affine flag variety} \label{affine-flag-variety}

Let $V = F^n$ with $G = \GL_n(F)$ acting on $V$ by left multiplication. Let $e_1, e_2, \dots, e_n$ denote the standard $F$-basis of $V$. For $c \in \ZZ$, we define $e_{i+nc} := t^ce_i$. Every element of $V$ can be uniquely written as an infinite $\K$-linear combination of $\{e_i\}_{i \in \ZZ}$, allowing an infinite number of terms with positive indices, but only allowing finitely many terms with negative indices.


An \defn{$A$-lattice} $\Lambda \subset V$ is an $A$-submodule given by $\Lambda = A v_1\oplus \cdots \oplus A v_n$, where ${v_1, \dots, v_n}$ is an $F$-basis of $V$. Each element of an $A$-lattice $\Lambda$ can be uniquely written as a formal $\K$-linear combination of the elements $v_i$ for $i\geq 1$, possibly involving infinitely many terms, where $v_{i+nc} := t^cv_i$.

The \defn{complete affine flag variety} $\widetilde{Fl}_n$ is defined as the space of all \defn{affine flags} of lattices $\Lambda_{\bullet} = (\Lambda_1 \supset \Lambda_2 \supset \dots \supset \Lambda_n)$ such that $\Lambda_n \supset t \Lambda_1$ and $\dim_{\K}(\Lambda_j / \Lambda_{j+1}) = 1$ for $j \in [1,n-1]$. 

\begin{example}
Define a family of $A$-lattices for $j \in \ZZ$:
\[ E_j :=  \text{span}_{A} \{ e_j, e_{j+1}, \dots, e_{j+n-1} \}. \]
Notice that $E_j = \sigma^jE_1$, where 
$ \sigma = \begin{psmallmatrix}  0 &0 & \cdots & 0 &t\\
				1 & 0 & \cdots & 0 &0\\
				0 & 1 & \cdots&0&0\\
				\vdots &\vdots &\ddots&\vdots & \vdots\\
				0 & 0 & \cdots & 1 & 0  \end{psmallmatrix}.$
The flag of lattices $(E_1 \supset E_2 \supset \dots \supset E_n)$ is referred to as the \defn{standard flag}.
\end{example}

The affine flag variety has been studied extensively for its relation with representation theory. Some examples of  applications are found in \cite{lam2021back, lee2019combinatorial, Nadler04}. A more extensive list of references can be found in Section 1.5 of \cite{tong2024orthogonalsymplecticorbitsaffine}.

The group $G$ acts on the affine flags $\Lambda_{\bullet}$ by $g\Lambda = (g\Lambda_1,g \Lambda_2, \dots, g\Lambda_n)$. Under the induced action on the affine flags modulo $t$, the stabilizer of $E_{\bullet}$ modulo $t$ is the group of lower triangular matrices in $\GL_n(\K)$. 
Therefore, under the action on the affine flags, the stabilizer of the standard flag $E_{\bullet}$ is the \defn{opposite Iwahori subgroup} $B'_n$ consisting of the matrices $b \in \GL_n(A)$ that are lower-triangular modulo $t$. Therefore, we have a bijection between $\widetilde{Fl}_n$ and $G/B'_n$ by the orbit-stabilizer theorem.

The following proposition shows that the elements of the affine flag variety are parametrized by ordered $F$-bases of $V$.

\begin{proposition}[{\cite[\S1.1]{Magyar}}]\label{affflagbasis}
For each affine flag $\Lambda_{\bullet} = (\Lambda_1 \supset \Lambda_2 \supset \dots \supset \Lambda_n)$, we can choose a vector $v_i \in \Lambda_i \setminus \Lambda_{i+1}$ for $1 \leq i \leq n-1$. Then the following properties hold:
\ben
\item[(i)] $\dim_{\K} (\Lambda_n / t\Lambda_1 ) = 1$, and so we can also select a vector $v_n \in \Lambda_n \setminus t\Lambda_1$.
\item[(ii)] The vectors $v_1, v_2, \dots, v_n$ form an $F$-basis of $F^n$.
\item[(iii)]  For $1 \leq i \leq n$, we have $\Lambda_i = \text{span}_{A}\{v_i, v_{i+1}, \dots, v_n, tv_1, \dots, tv_{i-1}\}$.
\een
\end{proposition}

\begin{remark}\label{gmod-two-b-bij-affine-flag}
There are natural bijections $G/B \to \widetilde{Fl}_n$ and $G/B' \to \widetilde{Fl}_n$, where $B$ and $B'$ are the Iwahori subgroups of matrices in $G$ that are respectively upper and lower triangular modulo $t$. 
The bijective map between $G/B$ and $\widetilde{Fl}_n$ is as follows:
for $gB \in G/B$, let $g_i$ be the $(n+1-i)$-th column of $g$.
The bijective map is
\begin{equation}\label{gmodb-bij-affine-flag}
gB \mapsto \Lambda_{\bullet}(g),
\end{equation}
where for $i \in [n]$, $\Lambda_{i}(g) = \text{span}_{A}\{g_i, g_{i+1}, \dots, g_n, tg_1, tg_2, \dots, tg_{i-1}\}$.

On the other hand, the bijective map between $G/B'$ and $\widetilde{Fl}_n$ is as follows: setting $g_i$ to be the $i$-th column of $g$, the bijective map is $gB' \mapsto \Lambda_{\bullet}(g)$, where $\Lambda_{i}(g) = \text{span}_{A}\{g_i, g_{i+1}, \dots, g_n, tg_1, tg_2, \dots, tg_{i-1}\}$.
\end{remark}

\section{Orbits of the product group $\GL_p \times \GL_q$}\label{chpt4-ref}
%
%

\subsection{$K$-invariants in affine flags}\label{sect4.4-ref}

Inspired by \cite{Yamamoto}, we discuss in this subsection some $K$-invariants of the dimensions of certain vector subspaces related to the affine flag under the $K$-action.

Recall that an \defn{affine $(p,q)$-clan} is a $\ZZ$-indexed sequence $c = (c_i)_{i \in \ZZ}$ with $n = p+q$
such that 
\ben
\item[(i)]each $c_i$ is either $+$, $-$ or an integer.
\item[(ii)]for $k \in \ZZ$, $c_{i+kn} = c_{i}+kn$ if $c_i$ is an integer and $c_{i+n} = c_i$ if $c_i$ is $+$ or $-$,
\item[(iii)] $\#\{ i \in [n] : c_i = +\} - \#\{ i \in [n] : c_i = -\} = p-q$,
and every integer, if it appears, appears exactly twice in the sequence.
\een

We do not distinguish between affine $(p,q)$-clans that are equivalent as described in Definition~\ref{affine-clan-def}.
Notice that by definition, an affine $(p,q)$-clan $c = (\dots,c_1,c_2,c_3,\dots)$ is uniquely determined by $(c_1, c_2, c_3, \dots c_n)$.

Every equivalence class of affine $(p,q)$-clans has a \defn{standard numbering} $\ZZ \cup \{+, - \}$-indexed sequence $c = (\dots,c_1,c_2,c_3,\dots)$, satisfying the following:
\ben
\item[(i)] If $i<j$, $i \in [n]$ and $c_i = c_j = a \in \ZZ$, then $c_i \in [n]$.
\item[(ii)] If $i<j$, $j \in [n]$, $i <1$ and $c_i = c_j = a \in \ZZ$, then $c_j \in [n]$.
\item[(iii)] For $i \in [n]$, if $c_i \in \ZZ$ and $c_1, c_2, \dots, c_{i-1} \in \{+, - \}$, then $c_i = 1$.
\item[(iv)] For $i,j \in [n]$, if $c_i, c_j \in \ZZ$, $c_{i+1}, c_{i+2}, \dots, c_{j-1} \in \{+, - \}$ and $c_j \neq c_{i'}$ for some $1\leq i' < i$, then $c_j = c_i + 1$. 
\een

For example, for the case $n = 8$ and $p=5$, $q=3$, $c = (\dots, +, 1, -, 2, 1, +, 18, +,  \dots )$ is the standard numbering of its equivalence class, while $d = (\dots, +, 2, -, -15 , 2, +, 1, +,  \dots )$ is equivalent to $c$, but not a standard numbering.


Let $V_+$ and $V_-$ be the subspaces in $V = \KLt^n$ consisting of the vectors in $\KLt^n$ with the last $q$ coordinates equal to zero and the first $p$ coordinates equal to zero, respectively. Define $\pi_+$ and $\pi_-$ to be the projection maps onto $V_+$ and $V_-$ respectively from the decomposition $V = V_+ \oplus V_-$.

We now define some important $K$-invariant data for any affine flag.

\begin{definition}\label{affine-dim-inv-def}
For an affine flag $\Lambda_{\bullet} = (\Lambda_1 \supset \Lambda_2 \dots \supset \Lambda_n)$, and $i \in \ZZ$, define,
\begin{align*}
(i; +) &= \dim_{\K} ((\Lambda_i \cap V_+) / (\Lambda_{i+1} \cap V_+)),\\
(i; -) &= \dim_{\K} ((\Lambda_i \cap V_-) / (\Lambda_{i+1} \cap V_-)),\\
(i; \NN) &= \dim_{\K} (\Lambda_i /((\Lambda_i \cap V_+) \oplus (\Lambda_i \cap V_-))),
\end{align*}
and for $i, j \in \ZZ$ satisfying $i\leq j$ and $i\leq n+1$, define
\begin{align*}
(i;j) & =  \dim_{\K}( (\pi_+(\Lambda_j) + \Lambda_i)/t\Lambda_1 ).
\end{align*}
\end{definition}

First we have to make sure the above quantities are finite non-negative integers.

\begin{lemma}\label{across-pos-neg-finite-dim}
The quantities $(i; \NN) = \dim_{\K} (\Lambda_i /((\Lambda_i \cap V_+) \oplus (\Lambda_i \cap V_-)))$ are finite. 
\begin{proof}
By Proposition \ref{affflagbasis} we can write 
\[\Lambda_i = \text{span}_{A} \{v_i, v_{i+1} \dots, v_n, tv_1, tv_2, \dots, tv_{i-1}\},\]
where $v_1, v_2, \dots, v_n$ is an $F$-basis for $F^n$. Therefore for each $j \in [n]$, there exists a minimum $N_j \in \ZZ$ such that 
\[t^{N_j} e_j = a_{ji} v_i + a_{j(i+1)}v_{i+1} + \dots + a_{jn}v_n + a_{j1}tv_1 + \dots + a_{j(i-1)}tv_{i-1},\]
 where $a_{j1}, a_{j2}, \dots, a_{jn} \in A$. In other words, for $M \geq N_j$, it holds that $t^Me_j \in \Lambda_i \cap V_+$ for $1\leq j \leq p$, while it holds that $t^Me_j \in \Lambda_i \cap V_-$ for $p+1 \leq j \leq n$. Therefore if $M_j$ is the smallest order of the $j$-th entry of vectors in $\Lambda_i$, the maximum dimension of $\dim_{\K} (\Lambda_i /((\Lambda_i \cap V_+) \oplus (\Lambda_i \cap V_-)))$ is $\sum^{N}_{j=1} (N_j-M_j)$, which is finite.
\end{proof}
\end{lemma}

\begin{lemma} \label{lattice-intersect-free}
For $i \in \ZZ$, the $A$-modules $\Lambda_i \cap V_+$ and $\Lambda_i \cap V_-$ are free $A$-modules of rank $p$ and $q$ respectively.
\end{lemma}

\begin{proof}
Notice that $A$ is a discrete valuation ring, and hence a principal ideal domain.  
By \cite[Theorem 9.8]{rotman2010advanced},
every submodule of a finite-dimensional free $A$-module is free. Now notice that since $\Lambda_i$ is a free $A$-module of rank $n$, the $A$-submodules $\Lambda_i \cap V_+$ and  $\Lambda_i \cap V_-$ of $\Lambda_i$ are both free.

It remains to show that the ranks of $\Lambda_i \cap V_+$ and $\Lambda_i \cap V_-$ are $p$ and $q$ respectively. We now prove that $\Lambda_i \cap V_+$ has rank $p$. The proof for the rank of $\Lambda_i \cap V_-$ being $q$ is similar. 

Notice that since $\Lambda_i$ has a basis that is also an $F^n$ basis, each $e_j$ for $1\leq j \leq n$ is an $F$-linear combination of the basis elements of $\Lambda_i$. Therefore for each $e_j$, there is a large enough integer $N_j$  such that $t^{N_j} e_j \in \Lambda_i$, and hence 
\[\text{span}_{A}\{t^{N_1} e_1, t^{N_2} e_2 , \dots, t^{N_p} e_p\} \subset \Lambda_i \cap V_+\] 
as a submodule. It follows that  $\Lambda_i \cap V_+$ has rank at least $p$.

On the other hand, write $\Lambda_i = \text{span}_{A} \{v_i, v_{i+1} \dots, v_n, tv_1, tv_2, \dots, tv_{i-1}\}$, and for $1\leq j \leq n$, set $M_j \in \ZZ$ to be the smallest order of the $j$-th entry in vectors $\{v_i, v_{i+1} \dots, v_n, tv_1, tv_2, \dots, tv_{i-1}\}$. Notice that this is also the smallest order of the $j$-th entry in all the vectors in $\Lambda_i$. We have
\[ \Lambda_i \cap V_+ \subset \text{span}_{A}\{t^{M_1} e_1, t^{M_2} e_2 , \dots, t^{M_p} e_p\}, \]
since every element in $\Lambda_i \cap V_+$ is an $A$-linear combination of $t^{M_1} e_1, t^{M_2} e_2 , \dots, t^{M_p} e_p$. Therefore $\Lambda_i \cap V_+$ has rank at most $p$, and hence equal to $p$.
\end{proof}

\begin{lemma} \label{plus-minus-takes-01}
For each $i \in \ZZ$, it holds that $(i;+)$, $(i;-) \in \{0, 1\}$.
\end{lemma}
\begin{proof}
We show this fact for $(i;+)$ only. The proof for $(i;-) \in \{0, 1\}$ is similar. 

By Lemma~\ref{lattice-intersect-free}, we can write $\Lambda_i \cap V_+ = \text{span}_{A}\{w_1, w_2, \dots w_p \}$ for some $w_1, w_2, \dots w_p \in F^n$. 
For $j \in [1,p]$, we have 
\begin{align*}
w_j &= a_{ji} v_i + a_{j(i+1)}v_{i+1} + \dots + a_{jn}v_n + a_{j1}tv_1 + \dots + a_{j(i-1)}tv_{i-1} \\
&= a_{j(i+1)}v_{i+1} + \dots + a_{jn}v_n + a_{j1}tv_1 + \dots + a_{j(i-1)}tv_{i-1} + (t^{-1}a_{ji})(tv_i),
\end{align*} 
where $a_{j1}, a_{j2}, \dots, a_{jn} \in A$. 

If none of the $a_{ji}$'s are in $A^*$ for $j \in [1,p]$, then all $a_{ji} \in tA$. Therefore $t^{-1}a_{ji} \in A$ for all $j \in [1,p]$, and hence $\Lambda_i \cap V_+ = \text{span}_{A}\{w_1, w_2, \dots w_p \} \subset \Lambda_{i+1} \cap V_+ \subset \Lambda_i \cap V_+$. It follows that $\Lambda_i \cap V_+ = \Lambda_{i+1} \cap V_+$ and $(i;+) = \dim_{\K} ((\Lambda_i \cap V_+) / (\Lambda_{i+1} \cap V_+)) = 0$.

Otherwise suppose there are $j$'s such that $a_{ji} \in A^*$. If there are two or more $j$'s such that $a_{ji} \in A^*$, pick a $j$, say $j_1 \in [1, p]$, and eliminate the constant terms of the coefficients of $v_i$ in other such $j$'s by subtracting a non-zero $\K$-multiple of $w_{j_1}$. 
Therefore without loss of generality we can assume that there is only one $w_j$ within $w_1, w_2, \dots w_p$ such that $a_{ji} \in A^*$ in the above expansions as $A$-linear combinations of $\{v_i, v_{i+1}, \dots, v_n, tv_1, tv_2, \dots, tv_{i-1}\}$. 
In particular for this $w_j$, it holds that $w_j \notin \Lambda_{i+1} \cap V_+$. 
On the other hand, we have $tw_j \in \Lambda_{i+1} \cap V_+$, since
\begin{align*}
tw_j &= ta_{j(i+1)}v_{i+1} + \dots + ta_{jn}v_n + ta_{j1}tv_1 + \dots + ta_{j(i-1)}tv_{i-1} + (a_{ji})(tv_i),
\end{align*} 
while $w_1, w_2, \dots, w_{j-1}, w_{j+1}, \dots w_p \in \Lambda_{i+1} \cap V_+$ by assumption. Therefore,
\begin{align*}
\text{span}_{A}\{w_1, w_2, \dots, w_{j-1}, tw_j, w_{j+1}, \dots, w_p \} &\subset \Lambda_{i+1} \cap V_+ \\
& \subsetneq \Lambda_i \cap V_+ = \text{span}_{A}\{w_1, w_2, \dots w_p \} .
\end{align*}
Observe that, 
\begin{align*}
0 < (i;+) &= \dim_{\K} \( \dfrac{\Lambda_i \cap V_+}{\Lambda_{i+1} \cap V_+} \) \\
& \leq \dim_{\K} \( \dfrac{\text{span}_{A}\{w_1, w_2, \dots w_p \}}{\text{span}_{A}\{w_1, w_2, \dots, w_{j-1}, tw_j, w_{j+1}, \dots, w_p \}} \) = 1.
\end{align*}
Therefore it holds that  $(i;+) = 1$. 
\end{proof}

\begin{lemma}
The quantities $(i;j) = \dim_{\K}((\pi_+(\Lambda_j)+ \Lambda_i)/t\Lambda_1)$ are finite for $i\leq j$ and $i\leq n+1$.
\begin{proof}
Suppose $j = j' + an$, where $1\leq j' \leq n$ and $j', a \in \ZZ$. 
Write 
\[\Lambda_{j'} = \text{span}_{A}\{v_{j'}, v_{j'+1}, \dots, v_n, tv_1, tv_2, \dots, tv_{j'-1}\}.\]
Then
\[\Lambda_j = t^a \Lambda_{j'} = \text{span}_{A}\{t^a v_{j'}, t^a v_{j'+1}, \dots, t^a v_n, t^{a+1} v_1, t^{a+1} v_2, \dots, t^{a+1} v_{j'-1}\},\]
and therefore
\begin{multline*}
\pi_+(\Lambda_j) = \text{span}_{A}\{t^a \pi_+(v_{j'}), t^a \pi_+(v_{j'+1}), \dots, \\
t^a \pi_+(v_n), t^{a+1} \pi_+(v_1), t^{a+1} \pi_+(v_2), \dots,
 t^{a+1} \pi_+(v_{j'-1})\}.
\end{multline*}
Each of the $n$ vectors in the spanning set of $\pi_+(\Lambda_j)$ is an $F$-linear combination of $v_1, v_2, \dots, v_n$, and hence all lie in $\Lambda_m$ for small enough $m \in \ZZ$.
Therefore $\pi_+(\Lambda_i) \subset \Lambda_m$ as a $\K$-vector subspace for small enough $m$, in particular for $m \leq j$. Now we have 
\[ (\pi_+(\Lambda_i)+ \Lambda_j)/t\Lambda_i \subset (\Lambda_m + \Lambda_j)/t\Lambda_i = \Lambda_m/ t\Lambda_i.\]
Since $\dim_{\K}(\Lambda_m/ t\Lambda_i) = n+1-m$, it holds that $(i;j) = \dim_{\K} ((\pi_+(\Lambda_i)+ \Lambda_j)/t\Lambda_i ) \leq n+1-m$, and hence this number is finite.
\end{proof}
\end{lemma}

Recall that for any affine flag $\Lambda_{\bullet} = (\Lambda_1 \supset \Lambda_2 \supset \dots \supset \Lambda_n)$ and $g \in G$, we have $g\Lambda_{\bullet} = (g\Lambda_1 \supset g\Lambda_2 \supset \dots \supset g\Lambda_n)$.
By noticing that $K$ contains invertible matrices $g$ such that $gV_+ = V_+$, $gV_- = V_-$, we can prove the following statement. 
\begin{lemma} \label{k-inv-quantities}
Suppose that $\Lambda_{\bullet}$ is an affine flag, and $k \in K$.
For $i \in [n]$, $j \in \ZZ$ and $j \geq i$, the numbers $(i; +)$, $(i; -)$, $(i; \NN)$ and $(i;j)$ are the same when defined relative to $\Lambda_{\bullet}$ and $k\Lambda_{\bullet}$.
\begin{proof}
Notice that 
\begin{align*} \dim_{\K} (k\Lambda_i /((k\Lambda_i \cap V_+) \oplus (k\Lambda_i \cap V_-))) &= \dim_{\K} (\Lambda_i /((\Lambda_i \cap k^{-1}V_+) \oplus (\Lambda_i \cap k^{-1}V_-)))\\
& = \dim_{\K} (\Lambda_i /((\Lambda_i \cap V_+) \oplus (\Lambda_i \cap V_-))).
\end{align*} 
So the numbers $(i; \NN)$ are the same when defined relative to $\Lambda_{\bullet}$ and $k\Lambda_{\bullet}$. 

Similarly, 
\begin{align*}
\dim_{\K} ((k\Lambda_i \cap V_+) / (k\Lambda_{i+1} \cap V_+)) &= \dim_{\K} (( \Lambda_i \cap k^{-1} V_+) / (\Lambda_{i+1} \cap k^{-1} V_+)) \\
&= \dim_{\K} ((\Lambda_i \cap V_+) / (\Lambda_{i+1} \cap V_+)).
\end{align*} 
So the numbers $(i;+)$ and similarly $(i;-)$ are the same when defined relative to $\Lambda_{\bullet}$ and $k\Lambda_{\bullet}$.

Finally notice that $\pi_+(\Lambda_j) + \Lambda_i = ((\Lambda_j + V_-)\cap V_+) + \Lambda_i$. We have for $k \in K$,
\begin{align*}
\dim_{\K}( (\pi_+(k\Lambda_j) + k\Lambda_i)/tk\Lambda_1) & = \dim_{\K}( (k\Lambda_j + V_-) \cap V_+ + k\Lambda_i)/tk\Lambda_1) \\
&= \dim_{\K}( (\Lambda_j + k^{-1}V_-) \cap k^{-1}V_+ + \Lambda_i)/t\Lambda_1)\\
&= \dim_{\K}( (\Lambda_j + V_-) \cap V_+ + \Lambda_i)/t\Lambda_1)\\
&=  \dim_{\K}( \pi_+(\Lambda_j) + \Lambda_i)/t\Lambda_1)\\
&= (i;j).
\end{align*}
So the numbers $(i;j)$ are the same when defined relative to $\Lambda_{\bullet}$ and $k\Lambda_{\bullet}$.
\end{proof}
\end{lemma}

By Lemma~\ref{k-inv-quantities}, we say that $(i; +)$, $(i; -)$, $(i; \NN)$ and $(i;j)$ are \defn{left $K$-invariants}.

\subsection{Lemmas related to the left $K$-invariants}\label{sect4.5-ref}
In this section we record some technical lemmas on the left $K$-invariants of an affine flag $\Lambda_{\bullet}$ that will be used in Definition~\ref{alg-affine-flag-to-affine-clan} and the proof of Proposition~\ref{affine-flag-to-affine-clan}. 

The following lemma is analogous to the fact used in \cite[Prop. 2.2.6]{Yamamoto}. Recall that for an affine flag $\Lambda_{\bullet}$,  we defined $(i;\NN) = \dim_{\K} \left( \dfrac{\Lambda_{i}}{(\Lambda_{i} \cap V_+) \oplus (\Lambda_{i} \cap V_-)} \right)$ for any $i \in \ZZ$.

\begin{lemma} \label{k-inv-dim-rel}
For $i \in \ZZ$, it holds that $(i;\NN) - (i+1;\NN) =1 - (i;+) - (i;-)$.
\end{lemma}

\begin{proof}
Treat $\Lambda_i, \Lambda_{i+1}, \Lambda_i \cap V_+, \Lambda_i \cap V_-, \Lambda_{i+1} \cap V_+,$ and  $\Lambda_{i+1} \cap V_-$ as $\K$-vector spaces. By the third isomorphism theorem for $\K$-vector spaces,
\[ \frac{\Lambda_{i}}{(\Lambda_{i+1} \cap V_+) \oplus (\Lambda_{i+1} \cap V_-)}  \Bigg/ \frac{(\Lambda_i \cap V_+) \oplus (\Lambda_i \cap V_-)}{(\Lambda_{i+1} \cap V_+) \oplus (\Lambda_{i+1} \cap V_-)} \cong \frac{\Lambda_{i}}{(\Lambda_{i} \cap V_+) \oplus (\Lambda_{i} \cap V_-)}. \]
But since the quotients 
\[ \frac{(\Lambda_i \cap V_+) \oplus (\Lambda_i \cap V_-)}{(\Lambda_{i+1} \cap V_+)\oplus (\Lambda_{i+1} \cap V_-)} \cong \frac{(\Lambda_i \cap V_+)}{(\Lambda_{i+1} \cap V_+)} \oplus \frac{(\Lambda_i \cap V_-)}{(\Lambda_{i+1} \cap V_-)}   \text{ and }
\frac{\Lambda_{i}}{(\Lambda_{i} \cap V_+) \oplus (\Lambda_{i} \cap V_-)} \]
are finite-dimensional $\K$-vector spaces with dimensions $(i;+) + (i;-)$ and $(i;\NN)$ respectively, it holds that $\dfrac{\Lambda_{i}}{(\Lambda_{i+1} \cap V_+) \oplus (\Lambda_{i+1} \cap V_-)}$ is also a finite-dimensional $\K$-vector space, with dimension equal to $(i;\NN) + (i;+) + (i;-)$. 

By the third isomorphism theorem for $\K$-vector spaces again, we have
\[ \frac{\Lambda_{i}}{(\Lambda_{i+1} \cap V_+) \oplus (\Lambda_{i+1} \cap V_-)}  \Bigg/  \frac{\Lambda_{i+1}}{(\Lambda_{i+1} \cap V_+) \oplus (\Lambda_{i+1} \cap V_-)} \cong \frac{\Lambda_i}{\Lambda_{i+1}}.\]
Therefore combining the facts that $\dim_{\K} \left(\dfrac{\Lambda_{i}}{(\Lambda_{i+1} \cap V_+) \oplus (\Lambda_{i+1} \cap V_-)} \right) = (i;\NN) + (i;+) + (i;-)$, and that $\dim_{\K} (\Lambda_i / \Lambda_{i+1}) = 1$, we have
\[(i;\NN) + (i;+) + (i;-) - (i+1; \NN) = 1.\]
\end{proof}

Together with Lemma~\ref{plus-minus-takes-01}, the following corollary holds.

\begin{corollary}
It holds that $(i;\NN) - (i+1;\NN) \in \{1, 0, -1\}$ for all $i \in [n]$.
\end{corollary}

\begin{lemma}\label{affine-proj-dim}
For all $i \in \ZZ$, it holds that $\dim_{\K} \(\dfrac{\pi_+(\Lambda_i)}{\pi_+(\Lambda_{i+1})}\) = 1 - (i;-)$ and $\dim_{\K} \(\dfrac{\pi_-(\Lambda_i)}{\pi_-(\Lambda_{i+1})}\) = 1 - (i;+)$.
\end{lemma}

\begin{proof}
Notice that $\Lambda_i \cap V_+ = \pi_+(\Lambda_i \cap V_+) \subset \pi_+(\Lambda_i)$. Therefore $\pi_+(\Lambda_i) / (\Lambda_i \cap V_+)$ is a well-defined $\K$-vector space.

We treat $V_+$ and $\Lambda_i \cap V_+$ as $\K$-vector spaces, and consider the composition of the canonical quotient map $V_+ \rightarrow V_+/(\Lambda_i \cap V_+)$ with the map $\pi_+: \Lambda_i \rightarrow V_+$. This composition map has image equal to $\pi_+(\Lambda_i)/(\Lambda_i \cap V_+)$ and kernel equal to $(\Lambda_i \cap V_-) \oplus (\Lambda_i \cap V_+)$. By the first isomorphism theorem of $\K$-vector spaces, we have
\[ \frac{\Lambda_i}{(\Lambda_i \cap V_-) \oplus (\Lambda_i \cap V_+)} \cong \frac{\pi_+(\Lambda_i)}{\Lambda_i \cap V_+}. \]

Since $ \dim_{\K} \left( \dfrac{\Lambda_i}{(\Lambda_i \cap V_-) \oplus (\Lambda_i \cap V_+)} \right) = (i;\NN)$, it holds that $\dim_{\K} \left(\dfrac{\pi_+(\Lambda_i)}{\Lambda_i \cap V_+} \right)= (i;\NN)$ by the above isomorphism.

Similarly, the quotient $\dfrac{\pi_+(\Lambda_{i+1}) }{ \Lambda_{i+1} \cap V_+}$ is a finite-dimensional $\K$-vector space and $\dim_{\K} \(\dfrac{\pi_+(\Lambda_{i+1}) }{ \Lambda_{i+1} \cap V_+}\) = \dim_{\K} \(\dfrac{\Lambda_{i+1}}{(\Lambda_{i+1} \cap V_-) \oplus (\Lambda_{i+1} \cap V_+) } \)= (i+1;\NN)$. 

Notice also that by the third isomorphism theorem on $\K$-vector spaces, we have
\[ \frac{\pi_+(\Lambda_i)}{\Lambda_{i} \cap V_+} \cong \frac{\pi_+(\Lambda_i)}{\Lambda_{i+1} \cap V_+} \Bigg/ \frac {\Lambda_{i} \cap V_+}{\Lambda_{i+1} \cap V_+}.\]
Since $\dfrac{\pi_+(\Lambda_i)}{\Lambda_{i} \cap V_+} $ and $\dfrac {\Lambda_{i} \cap V_+}{\Lambda_{i+1} \cap V_+}$ are finite-dimensional $\K$-vector spaces of dimension $(i;\NN)$ and $(i;+)$ respectively, $\dfrac{\pi_+(\Lambda_i)}{\Lambda_{i+1} \cap V_+}$ is also a finite-dimensional $\K$-vector space of dimension $(i;\NN) + (i ; +)$.

By the third isomorphism theorem on $\K$-vector space again, we have
\[ \frac{\pi_+(\Lambda_i)}{\pi_+(\Lambda_{i+1})} \cong \dfrac{\pi_+(\Lambda_i)}{\Lambda_{i+1} \cap V_+}\Bigg/ \frac{\pi_+(\Lambda_{i+1})} {\Lambda_{i+1} \cap V_+}. \] 
Therefore $\dfrac{\pi_+(\Lambda_i)}{\pi_+(\Lambda_{i+1})}$ is a finite-dimensional $\K$-vector space of dimension $(i;\NN) + (i;+) - (i+1;\NN)$, and by Lemma~\ref{k-inv-dim-rel},
\begin{align*}
\dim_{\K} \(\dfrac{\pi_+(\Lambda_i)}{\pi_+(\Lambda_{i+1})}\) & = (i;\NN) + (i;+) - (i+1; \NN)\\
& = 1-(i;+) - (i;-) + (i;+)\\
& = 1 - (i;-).
\end{align*}
The proof for $\dim_{\K} \( \dfrac{\pi_-(\Lambda_i)}{\pi_-(\Lambda_{i+1})} \) = 1 - (i;+)$ is similar.
\end{proof}

\begin{lemma}\label{affine-lmpli-overlap-dim-decr}
Suppose $i \in \ZZ$. If $(i;+) = 1$ and $(i;- ) = 1$, then there exist $v_+ \in \pi_+(\Lambda_{i+1})$ and $v_- \in \pi_-(\Lambda_{i+1})$ satisfying $v_+, v_- \notin \Lambda_{i+1}$, such that  $\Lambda_i = \Lambda_{i+1} \oplus \K v_+$ and $\Lambda_i = \Lambda_{i+1} \oplus \K v_-$. 
\end{lemma}
\begin{proof}
The fact that $(i;+) = \dim_{\K}( (\Lambda_i \cap V_+)/(\Lambda_{i+1} \cap V_+))= 1$ means that there exists $v_+ \in \Lambda_i \cap V_+$ such that 
\[ \Lambda_i \cap V_+ = (\Lambda_{i+1} \cap V_+) \oplus \K v_+.\]
Notice that $v_+ \in \Lambda_i$ and $v_+ \notin (\Lambda_{i+1} \cap V_+)$. But since $v_+ \in V_+$, we have $v_+ \notin \Lambda_{i+1}$, and therefore 
\[ \Lambda_i \supset \Lambda_{i+1} \oplus \K v_+ \supsetneq \Lambda_{i+1}. \] 
Since $\dim_{\K}(\Lambda_i/\Lambda_{i+1}) = 1$, we conclude that 
\[ \Lambda_i = \Lambda_{i+1} \oplus \K v_+.\]
Similar arguments using the fact $(i;-) = \dim_{\K}( (\Lambda_i \cap V_-)/(\Lambda_{i+1} \cap V_-))= 1$ show that there exists $v_- \in \pi_-(\Lambda_{i+1})$ satisfying
\[ \Lambda_i = \Lambda_{i+1} \oplus \K v_-.\]
It remains to show that $v_+ \in \pi_+(\Lambda_{i+1})$ and $v_- \in \pi_-(\Lambda_{i+1})$.
In fact, it holds that  $\pi_+(\Lambda_{i}) = \pi_+(\Lambda_{i+1})$ and $\pi_-(\Lambda_{i}) = \pi_-(\Lambda_{i+1})$, since by Lemma~\ref{affine-proj-dim} we have 
\[\dim_{\K}( \pi_+(\Lambda_{i})/ \pi_+(\Lambda_{i+1})) = 1-(i;-) = 0\] 
and
 \[\dim_{\K}( \pi_-(\Lambda_{i})/ \pi_-(\Lambda_{i+1})) = 1- (i;+) = 0.\] 
However, we have $v_+ \in \Lambda_i \cap V_+ =  \pi_+ (\Lambda_i \cap V_+) \subset \pi_+(\Lambda_i) = \pi_+(\Lambda_{i+1})$ and $v_- \in \Lambda_i \cap V_- = \pi_+(\Lambda_i \cap V_-)\subset \pi_-(\Lambda_i)= \pi_-(\Lambda_{i+1})$. 
\end{proof}

\begin{lemma}\label{affine-find-left-bracket}
Suppose $v_+ \in \pi_+(\Lambda_{i+1})$ and $v_+ \notin \Lambda_{i+1}$. Then there exists a unique $j > i$ such that $v_+ \in \pi_+(\Lambda_j) + \Lambda_{i+1}$ and $v_+ \notin \pi_+(\Lambda_{j+1}) + \Lambda_{i+1}$.
\begin{proof}
Suppose not. Using the fact that $v_+ \in \pi_+(\Lambda_{i+1}) \subset \pi_+(\Lambda_{i+1}) + \Lambda_{i+1}$, we have $v_+ \in \pi_+(\Lambda_{j}) + \Lambda_{i+1}$ for all integers $j > i$. However, it holds that $t^N \pi_+(\Lambda_{i+1}) \subset \Lambda_{i+1}$ for large enough $N$, since $\Lambda_{i+1}$ is an $A$-lattice. Therefore, we have $v_+ \in \pi_+(\Lambda_{i+1 + Nn}) + \Lambda_{i+1} = \Lambda_{i+1}$, a contradiction to the assumption that $v_+ \notin  \Lambda_{i+1}$. 
\end{proof}
\end{lemma}

\begin{lemma}\label{across-pos-neg-diff-formula}
It holds that $ (i; \NN) = \dim_{\K}\( \dfrac{\pi_+(\Lambda_i) + \Lambda_i}{\Lambda_i} \) $ for all $i \in \ZZ$.
\begin{proof}
We shall prove that $\dfrac{\Lambda_i}{(\Lambda_i \cap V_+) \oplus (\Lambda_i \cap V_-)}  \cong \dfrac{\pi_+(\Lambda_i) + \Lambda_i}{\Lambda_i}$ as $\K$-vector spaces.

Define the $\K$-linear transformation $f: \Lambda_i \rightarrow \dfrac{\pi_+(\Lambda_i) + \Lambda_i}{\Lambda_i}$ by $f(v) = \pi_+(v) + \Lambda_i$, where $v \in \Lambda_i$. Since $\pi_+$ is a linear transformation, the function $f$ is also a linear transformation between the two $\K$-vector spaces. It is clear that $f$ is surjective, by noticing that $\pi_+(v) + w + \Lambda_i = \pi_+(v) + \Lambda_i = f(v)$ for any $v, w \in \Lambda_i$. It remains to show that the kernel of $f$ is $(\Lambda_i \cap V_+) \oplus (\Lambda_i \cap V_-)$.

Write 
\begin{equation}\label{expand-lattice}
\Lambda_i = (\Lambda_i \cap V_+) \oplus (\Lambda_i \cap V_-) \oplus W_i, 
\end{equation}
 where $W_i$ is a complementary $\K$-subspace of $(\Lambda_i \cap V_+) \oplus (\Lambda_i \cap V_-)$ in $\Lambda_i$, and of finite $\K$-dimension by Lemma~\ref{across-pos-neg-finite-dim}. Then for any $v \in \Lambda_i$, it holds that $v = v_+ + v_- + w$, where $v_+ \in \Lambda_i \cap V_+$, $v_- \in \Lambda_i \cap V_-$, and $w \in W_i$. 

Notice that $f(v) = \pi_+(v) + \Lambda_i = \Lambda_i$ if and only if $\pi_+(v) \in \Lambda_i$. By writing $v = v_+ + v_- + w$, it holds that 
$\pi_+(v) = \pi_+(v_+ + v_- + w) = \pi_+(v_+) + \pi_+(v_-) + \pi_+(w) = v_+ + \pi_+(w) \in \Lambda_i$.
Since $v_+ \in \Lambda_i$, we have $\pi_+(w) \in \Lambda_i$. We shall now show that $w = 0$.

Write $w = w_+ + w_-$, where $w_+ \in V_+$ and $w_- \in V_-$, and notice that $\pi_+(w) = w_+ \in \Lambda_i$. It follows that $w_+ \in \Lambda_i \cap V_+$. Moreover, we have $w_- = w - w_+ \in \Lambda_i$ and therefore $w_- \in \Lambda_i \cap V_-$. It follows that $w = w_+ + w_- \in (\Lambda_i \cap V_+) \oplus (\Lambda_i \cap V_-)$. Since $W_i$ is a complementary $\K$-subspace of $(\Lambda_i \cap V_+) \oplus (\Lambda_i \cap V_-)$ in $\Lambda_i$ and $w \in W_i$, it holds that $w = 0$. 

As a result, the fact that $f(v) = \pi_+(v) + \Lambda_i = \Lambda_i$ implies that $v \in (\Lambda_i \cap V_+) \oplus (\Lambda_i \cap V_-)$, and, on the other hand, if $v \in (\Lambda_i \cap V_+) \oplus (\Lambda_i \cap V_-)$ then $f(v) = \pi_+(v) + \Lambda_i \in (\Lambda_i \cap V_+) + \Lambda_i \subset \Lambda_i$. Therefore it holds that $\ker f = (\Lambda_i \cap V_+) \oplus (\Lambda_i \cap V_-)$, and by the first isomorphism theorem for $\K$-vector spaces, we have
\[ \dfrac{\Lambda_i}{(\Lambda_i \cap V_+) \oplus (\Lambda_i \cap V_-)}  \cong \dfrac{\pi_+(\Lambda_i) + \Lambda_i}{\Lambda_i}.\]
It follows that $\dim_{\K} \( \dfrac{\pi_+(\Lambda_i) + \Lambda_i}{\Lambda_i} \) = \dim_{\K} \( \dfrac{\Lambda_i}{(\Lambda_i \cap V_+) \oplus (\Lambda_i \cap V_-)} \) = (i;\NN)$.
\end{proof}
\end{lemma}

\begin{corollary} \label{expand-lattice-cor}
It holds that $(i;i) = n+1-i + (i;\NN)$ for any $i \leq n+1$.
\begin{proof}
Using the third isomorphism theorem for $\K$-vector spaces, we have
\[ \dfrac{\pi_+(\Lambda_i) + \Lambda_i}{t\Lambda_1} \Bigg/ \dfrac{\Lambda_i}{t\Lambda_1} \cong \dfrac{\pi_+(\Lambda_i) + \Lambda_i}{\Lambda_i}. \]
Notice that $\Lambda_i/(t\Lambda_1)$ is a finite-dimensional $\K$-vector space of dimension $n+1-i$. By Lemma~\ref{across-pos-neg-diff-formula}, it holds that $\dfrac{\pi_+(\Lambda_i) + \Lambda_i}{\Lambda_i}$ is also a finite-dimensional $\K$-vector space of dimension $(i;\NN)$. Therefore the $\K$-vector space $\dfrac{\pi_+(\Lambda_i) + \Lambda_i}{t\Lambda_1}$ is also finite-dimensional, and $\dim_{\K} \left( \dfrac{\pi_+(\Lambda_i) + \Lambda_i}{t\Lambda_1} \right) = (i;i) = n+1-i + (i;\NN)$.
\end{proof}
\end{corollary}

The following lemma and corollary determines the base case for the backward induction used to prove (i) and (iii) in Proposition~\ref{affine-flag-to-affine-clan}.
\begin{lemma}
There exists $m \in \ZZ$ such that $\pi_+(\Lambda_m) \subset t\Lambda_1$. 
\end{lemma}
\begin{proof}
The result follows by arguments similar to the proof of Lemma~\ref{affine-find-left-bracket}.
\end{proof}

\begin{corollary}\label{induction-base-step-affine-clan}
There exists a unique $m \in \ZZ$ such that $\pi_+(\Lambda_m) \subset t\Lambda_1$ and $\pi_+(\Lambda_{m-1}) \not\subset t\Lambda_1$.
\end{corollary}

\subsection{Bijections between orbits of affine flags and affine clans}\label{sect4.6-ref}
The following is an algorithm whose input is an affine flag and whose output is a $\ZZ$-indexed sequence, which will later be shown to be an affine $(p,q)$-clan.
\begin{definition}\label{alg-affine-flag-to-affine-clan}
For an affine flag $\Lambda_{\bullet} = (\Lambda_1 \supset \Lambda_2 \dots \supset \Lambda_n)$, produce a $\ZZ$-indexed sequence $c = (\dots, c_1, c_2, \dots,  c_n, \dots)$ by the following inductive procedure:
\ben
\item[] Suppose $c_1, c_2, \dots, c_{i-1}$ are determined.
\item[(a)] If $c_i$ is already determined, then move on to determine $c_{i+1}$.
\item[(b)] If $(i;+) = 1 $ and $(i;- ) = 0$, set $c_i = +$. Set also $c_{i+rn} = +$ for $r \in \ZZ$.
\item[(c)] If  $(i;+) = 0$ and $(i;- ) = 1$, set $c_i = -$. Set also $c_{i+rn} = -$ for $r \in \ZZ$.
\item[(d)] If  $(i;+) = 0$ and $(i;- ) = 0$, then $(i;\NN) - (i+1;\NN) = 1$ by Lemma~\ref{k-inv-dim-rel}. Consider the set $\{c_1, c_2, \dots , c_{i-1}\} \cap [n]$. If the set is empty, set $c_i = 1$, and if the set is not empty, set $c_i = 1 + \max (\{c_1, c_2, \dots , c_{i-1}\} \cap [n])$. Set also $c_{i+rn} = c_i + rn$ for $r\in \ZZ$.
\item[(e)] If $(i;+) = 1$ and $(i;- ) = 1$, then $(i;\NN) - (i+1;\NN) = -1$ by Lemma~\ref{k-inv-dim-rel}. By Lemma~\ref{affine-lmpli-overlap-dim-decr} there exist $v_+ \in \pi_+(\Lambda_{i+1})$ and $v_- \in \pi_-(\Lambda_{i+1})$ satisfying $v_+, v_- \notin \Lambda_{i+1}$, such that $\Lambda_i = \Lambda_{i+1} \oplus \K v_+$ and $\Lambda_i = \Lambda_{i+1} \oplus \K v_-$.
By Lemma~\ref{affine-find-left-bracket}, there exists a unique $j > i$ such that $v_+ \in \pi_+(\Lambda_j) + \Lambda_{i+1}$ and $v_+ \notin \pi_+(\Lambda_{j+1}) + \Lambda_{i+1}$. (Here it is possible that $j > n$.)

If $j \textrm{ mod } n < i \textrm{ mod } n$ and hence $c_{j} \in \ZZ$ by Lemma~\ref{step-e-valid}, set $c_i = c_j$ and $c_{i+rn} = c_i + rn$ for $r \in \ZZ$.

If $j \textrm{ mod } n > i \textrm{ mod } n$ and hence $c_{j}$ is not determined yet,  consider the set $(\{c_1, c_2, \dots , c_{i-1}\} \cap [n])$ again. Set $c_j = c_i = 1$ if the set is empty, and set $c_j = c_i= 1+  \max (\{c_1, c_2, \dots , c_{i-1}\} \cap [n])$ if the set is not empty. In both cases set $c_{j+rn} = c_{i+rn} = c_i + rn$ for $r\in \ZZ$. 

Denote the indices $i+rn$ and $j+rn$ as being \defn{paired} for each $r \in \ZZ$.
\een
\end{definition}

The following are two technical lemmas related to Definition~\ref{alg-affine-flag-to-affine-clan}, which shows that $c_j \in \ZZ$ in case (e) if $c_j$ is already determined, and the $\ZZ$-indexed sequence produced in Definition~\ref{alg-affine-flag-to-affine-clan} is indeed an affine $(p,q)$-clan.
\begin{lemma}\label{step-e-valid}
For the $j\in \ZZ$ described in (e) in Definition~\ref{alg-affine-flag-to-affine-clan}, it holds that $(j; \NN) - (j-1; \NN) = 1$, implying that $(j;+)  = (j;-)  = 0$. 
\begin{proof}
We use the notations in case (e) of Definition~\ref{alg-affine-flag-to-affine-clan}. Write $v_+  = v_+' + v_+''$, where $v_+' \in \pi_+(\Lambda_j)$ and $v_+'' \in \Lambda_{i+1}$. Notice that $v_+' \notin \pi_+(\Lambda_{j+1})$, since $v_+ \notin \pi_+(\Lambda_{j+1}) + \Lambda_{i+1}$. 
Notice that
\[ \pi_+(\Lambda_j) = \pi_+(\Lambda_{j+1} \oplus \K v_j) = \pi_+(\Lambda_{j+1}) + \K \pi_+(v_j).\]
However, we have $v_+' \in \pi_+(\Lambda_j) \backslash \pi_+(\Lambda_{j+1})$. Therefore 
\[ \pi_+(\Lambda_{j+1}) \subsetneq \pi_+(\Lambda_{j+1}) \oplus \K v_+' \subset \pi_+(\Lambda_j), \]
and hence $\dim_{\K}(\pi_+(\Lambda_j)/\pi_+(\Lambda_{j+1})) = 1$ and 
\begin{align} \label{affine-pos-proj-incr-dim}
\pi_+(\Lambda_j) = \pi_+(\Lambda_{j+1}) \oplus \K v_+' .
\end{align}
By Lemma~\ref{affine-proj-dim}, we have $\dim_{\K}(\pi_+(\Lambda_j)/\pi_+(\Lambda_{j+1})) = 1 - (j;-)$, and therefore $(j;-) = 0$.
 Similarly, by the observation that $\pi_+(\Lambda_j) + \Lambda_{i+1} = \pi_-(\Lambda_j) + \Lambda_{i+1}$, it follows that $(j;+) = 0$. Therefore $c_j \in \ZZ$ by case (d) in the algorithm.
 \end{proof}
 \end{lemma}

\begin{lemma}\label{affirm-affine-clan}
The $\ZZ$-indexed sequence $c = (\dots, c_1, c_2, \dots, c_n, \dots)$ defined from $\Lambda_{\bullet}$ according to the procedure in Definition~\ref{alg-affine-flag-to-affine-clan} is an affine $(p,q)$-clan.
\begin{proof}
By construction, the $\ZZ$-indexed sequence $c$ has properties (i) and (ii) in Definition~\ref{affine-clan-def}. 
To show that any two indices in $[n]$ satisfying step (e) in Definition~\ref{alg-affine-flag-to-affine-clan} are not paired to indices in the same congruence class modulo $n$, suppose on the contrary that there are two indices $i_1 \neq i_2 \in [n]$ such that $(i_1; \NN) - (i_1+1; \NN) = (i_2; \NN) - (i_2+1; \NN) = -1$ and $c_{i_1} = c_{j_1} \in \ZZ$, $ c_{i_2} = c_{j_2} \in \ZZ$ for some $j_1 > i_1$ and $j_2 >  i_2$, and $j_1 \equiv j_2 \mod n$. Without loss of generality assume that $i_1>i_2$. By Lemma~\ref{affine-lmpli-overlap-dim-decr} there exist $v_{i_1,+}$ and $v_{i_2, +}$ such that $v_{i_1,+} \in \pi_+ (\Lambda_{j_1}) + \Lambda_{i_1+1}$ and $v_{i_2,+} \in \pi_+ (\Lambda_{j_2}) + \Lambda_{i_2+1}$. Write $v_{i_1,+} = v_{i_1,+}' + v_{i_1,+}''$ and $v_{i_2,+} = v_{i_2,+}' + v_{i_2,+}''$, where $ v_{i_1,+}' \in \pi_+(\Lambda_{j_1})$, $ v_{i_2,+}' \in \pi_+(\Lambda_{j_2})$, $v_{i_1,+}'' \in \Lambda_{i_1+1}$ and $v_{i_2,+}'' \in \Lambda_{i_2+1}$.
By \eqref{affine-pos-proj-incr-dim} we have
\[ \pi_+(\Lambda_{j_1}) = \pi_+(\Lambda_{j_1+1}) \oplus \K v_{i_1,+}' \quad \text{and} \quad \pi_+(\Lambda_{j_2})= \pi_+(\Lambda_{j_2+1}) \oplus \K v_{i_2,+}'. \]
Notice that since $j_1 \equiv j_2 \mod n$, we have $\Lambda_{j_1} = t^a \Lambda_{j_2}$ for some $a \in \ZZ$. Therefore it holds that $ \pi_+(\Lambda_{j_2})= \pi_+(\Lambda_{j_2+1}) \oplus \K t^a v_{i_1,+}'$ and there exists $w_+ \in \pi_+(\Lambda_{j_2+1})$ and $\ell \in \K$ such that $v_{i_2,+}' = w_+ + \ell t^a v_{i_1,+}'$. 
However, it holds that
\begin{align*}
v_{i_2, +} &=  v_{i_2,+}' + v_{i_2,+}'' \\
&= w_+ +  \ell t^a v_{i_1,+}' + v_{i_2,+}'' \\
&  \in \pi_+(\Lambda_{j_2+1}) + \Lambda_{i_1} + \Lambda_{i_2+1} = \pi_+(\Lambda_{j_2+1}) + \Lambda_{i_2+1},
\end{align*}
contradicting the definition of $j$ satisfying $v_{i_2, +} \notin \pi_+(\Lambda_{j+1}) + \Lambda_{i_2+1}$. 
Therefore for two indices $i_1 \neq i_2 \in [n]$ such that $(i_1; \NN) - (i_1-1; \NN) = (i_2; \NN) - (i_2-1; \NN) = -1$, and $c_{i_1}= c_{j_1}$ and $c_{i_2}  = c_{j_2}$ for some $j_1 > i_1$ and $j_2 > i_2$, it holds that $c_{j_1} \not\equiv c_{j_2} \textrm{ mod } n$. 

Now we proceed to show that $c$ satisfies property (iii) in Definition~\ref{affine-clan-def}. 
Notice that multiplications of integer powers of $t$ on both the numerators and denominators of the quotients 
\[ \dfrac{\Lambda_i \cap V_+}{\Lambda_{i+1} \cap V_+}, \quad \dfrac{\Lambda_i \cap V_-}{\Lambda_{i+1} \cap V_-}  \quand \dfrac{\Lambda_i}{(\Lambda_i \cap V_+) \oplus (\Lambda_i \cap V_-)}\]
 are invertible $\K$-linear transformation on the finite-dimensional $\K$-vector spaces. Therefore it holds that $ (i;+) = (i+an;+)$, $(i;-) = (i+an;-)$ and $(i;\NN) = (i+an;\NN)$ for $i \in [n]$ and $a\in \ZZ$, 
and it follows that $(1+n; \NN) - (1;\NN) = 0$. 

Consider the following telescoping sum:
 $0 = (1;\NN) - (1+n;\NN) = \displaystyle \sum_{i = 1}^n ((i;\NN) - (i+1;\NN)).$
Recall that $(i;\NN) - (i+1;\NN) = 0$ if and only if $c_i \in \{+, - \}$. Since different $c_i \in \ZZ$ with $(i; \NN) - (i+1; \NN)  =-1$ are paired to different $c_j = c_i \in \ZZ$ with $(j; \NN) - (j+1; \NN)  =1$, no $c_j \in \ZZ$ with $j \in [n]$ and $(j; \NN) - (j+1; \NN)  =1$ are left unpaired, since otherwise the telescoping sum above will be greater than $0$. 
Hence we conclude that each integer, if it appears in $c$, appears exactly twice in $c$.

Observe that by using the third isomorphism theorem successively, we obtain the isomorphism
\[ \dfrac{ \Lambda_1 \cap V_+} {t\Lambda_1 \cap V_+} \cong \dfrac{ \Lambda_1 \cap V_+} {\Lambda_2 \cap V_+} \oplus \dfrac{ \Lambda_2 \cap V_+} {\Lambda_3 \cap V_+} \oplus \cdots \oplus \dfrac{ \Lambda_n \cap V_+} {t\Lambda_1 \cap V_+}. \]
It follows that $\sum_{i=1}^n (i;+) = \dim_{\K}((\Lambda_1 \cap V_+) / (t\Lambda_1 \cap V_+)) = \dim_{\K}((\Lambda_1 \cap V_+) / (t(\Lambda_1 \cap V_+)) =  p$ by Lemma~\ref{lattice-intersect-free}, and similarly it holds that  $\sum_{i=1}^n (i;-) = q$. Moreover, the equality $(i;+) = 1$ holds only when either $c_i = +$ or $c_i \in \ZZ$ in the case $(i;+) = (i;-) =1$. Similarly  the equality $(i;-) = 1$ only when either $c_i = -$ or $c_i \in \ZZ$ in the case $(i;+) = (i;-) =1$. Therefore the number of $+$ minus the number of $-$ in $(c_1, c_2, \dots, c_n)$ is equal to $\sum_{i=1}^n (i;+)- \sum_{i=1}^n (i;-) = p-q$ as desired. 

We have shown that $(\dots, c_1, c_2, \dots, c_n, \dots)$ is indeed an affine $(p,q)$-clan. 
\end{proof}
\end{lemma}

The following proposition further describes the properties for such an affine $(p,q)$-clan produced in Definition~\ref{alg-affine-flag-to-affine-clan}. The additional properties involving $(i;+)$, $(i;-)$, $(i; \NN)$ and $(i;j)$ show that the affine flags in the same $K \cong \GL_p(F) \times \GL_q(F)$-orbit result in the same affine $(p,q)$-clan given by the algorithm in Definition~\ref{alg-affine-flag-to-affine-clan}.

\begin{proposition}\label{affine-flag-to-affine-clan}
For an affine flag $\Lambda_{\bullet} = (\Lambda_1 \supset \Lambda_2 \dots \supset \Lambda_n)$, there exists a unique affine $(p,q)$-clan $c(\Lambda_{\bullet}) = (\dots, c_1,c_2 \dots , c_n, \dots)$ such that for all $1 \leq i \leq n$ and $i < j$, the affine $(p,q)$-clan $c(\Lambda_{\bullet})$ satisfies:
\ben
\item[(i)] $(i; \NN) = \# \{a \in \ZZ : c_x = c_y = a \text{ for } x < i \leq y\}$;
\item[(ii)] if $c_i = +$, then $(i;+) = 1$ and $(i;-) = 0$; if $c_i = -$, then $(i;+)= 0$ and $(i;-) = 1$;
\item[(iii)] $(i;j) = n+1-i + \#\{ a \in \ZZ : c_x = c_y = a \text{ for } x < i < j \leq y\}$. 
\een
The unique affine $(p,q)$-clan is the one described in Definition~\ref{alg-affine-flag-to-affine-clan}, which we shall call the \defn{affine $(p,q)$-clan of $\Lambda_{\bullet}$}.
\end{proposition}

\begin{remark}
In fact, it holds that $(i; \NN) = \sum_{i=1}^n b_i$, where $b_i = |c_j - c_i| / n$ if there exists a unique $i<j\leq n$ with $c_i \equiv c_j \mod n$, and $b_i = 0$ if $c_i$ is $+$ or $-$.
\end{remark}

\begin{proof}

Notice that (ii) in Proposition~\ref{affine-flag-to-affine-clan} holds immediately by the construction of $c(\Lambda_{\bullet})$. 

We now proceed to show by induction that (i) and (iii) hold in Proposition~\ref{affine-flag-to-affine-clan}.
%
Corollary~\ref{expand-lattice-cor} states that $(i;i) = n+1-i + (i;\NN)$ for any $i \leq n+1$.
Therefore, Property (i) in Proposition~\ref{affine-flag-to-affine-clan} can be reformulated as \ben
\item[(i$'$)] For all $1 \leq i \leq n$, $(i;i) = n+1-i + \# \{a \in \ZZ : c_x = c_y = a \text{ for } x < i \leq y\}$.
\een
By Corollary~\ref{induction-base-step-affine-clan}, there exist a unique $m \in \ZZ$ such that $\pi_+(\Lambda_m) \subset t\Lambda_1$ and $\pi_+(\Lambda_{m-1}) \not\subset t\Lambda_1$.
We prove properties (i) and (iii) in Proposition~\ref{affine-flag-to-affine-clan} by backward induction on $m$. First we prove that $(i;j) = n+1-i + \#\{ a \in \ZZ : c_x = c_y = a \text{ for } x < i < j \leq y\}$ for $i \in [n]$ and $j \geq m$ as the base case.

Case (a): Suppose that $1 \leq m \leq n$. Then the index $j$ involved in case (e) of Definition~\ref{alg-affine-flag-to-affine-clan} must be smaller than $m$, since if otherwise $j \geq m$, then it holds that $v_+ \in \pi_+(\Lambda_j) + \Lambda_{i+1} \subset t\Lambda_1 + \Lambda_{i+1} = \Lambda_{i+1} \subset \pi_+(\Lambda_{j+1}) + \Lambda_{i+1}$, a contradiction. 
Therefore all integers in $c_1, c_2, \dots, c_{m-1}$ cannot be paired with integers other than those $c_1, c_2, \dots, c_{m-1}$. It follows that for $j \geq m$, and $i \leq n$, we have 
$\{a\in \ZZ : c_x = c_y = a \in \ZZ \text{ for } x < i < j \leq y\} = 0$
 and  
 \begin{align*}
 (i;j) & = \dim_{\K}((\pi_+(\Lambda_j)+ \Lambda_i)/t\Lambda_1) = \dim_{\K}(\Lambda_i/t\Lambda_1)  \\
 & = n+1-i = n + 1 - i + \{a\in \ZZ : c_x = c_y = a \in \ZZ \text{ for } x < i < j \leq y\}.
 \end{align*}

Case (b): Suppose $m > n$.
Since $\pi_+(\Lambda_m) \subset t\Lambda_1$ and $\pi_+(\Lambda_{m-1}) \not\subset t\Lambda_1$, it holds that $\pi_+(\Lambda_m) \subsetneq \pi_+(\Lambda_{m-1})$. Hence by Lemma~\ref{affine-proj-dim} and \ref{plus-minus-takes-01}, it holds that $\dim_{\K}(\pi_+(\Lambda_{m-1})/\pi_+(\Lambda_m)) = 1 -(m-1;-) = 1$. It follows that $(m-1;-) = 0$. Moreover by the observation $\pi_+(\Lambda_j) + t\Lambda_1 = \pi_-(\Lambda_j) + t\Lambda_1$ for $j \geq n+1$, it holds that $\pi_+(\Lambda_m) \subset t\Lambda_1$ if and only if $\pi_-(\Lambda_m) \subset t\Lambda_1$. Therefore similar arguments yield $(m-1;+) = 0$. It follows that $c_{m-1} \in \ZZ$ and is paired to $c_{m'} = c_{m-1}$ for some $m' < m -1$. 

We shall show that $m' \geq 1$. If $m' < 1$, then $c_{m' \textrm{ mod } n} = c_{m-1 + an}$ for some $a \in \ZZ_{\geq 0}$. However   this is impossible, since case (e) of Definition~\ref{alg-affine-flag-to-affine-clan} would imply that $\pi_+(\Lambda_{m-1 + an}) \not\subset t\Lambda_1$, a contradiction to the definition of $m$. 

Now observe that for $i \in [n]$ and $j \geq m$, it holds that $(i;j) = n+1-i$ by the same computation as in case (a) above, while by definition of $m$, the index $m-1$ is paired to an index $i$ with $i \geq 1$. Moreover, since $\pi_+(\Lambda_m) \subset t\Lambda_1 \subset \Lambda_i$ for $i \in [n]$, we have $\pi_+(\Lambda_j) + \Lambda_i = \Lambda_i$ for $j \geq m$. It follows that in case (e) of Definition~\ref{alg-affine-flag-to-affine-clan}, the index $m-1$ is the largest possible index to be paired with an $i \in [n]$. Therefore again it holds that $\{a\in \ZZ : c_x = c_y = a \in \ZZ \text{ for } x < i < j \leq y\} = 0$ for $j \geq m$ and $i \in [n]$, and hence $(i;j) = n+1 - i = n + 1 - i + \{a\in \ZZ : c_x = c_y = a \in \ZZ \text{ for } x < i < j \leq y\} $.

Thus the base case that $(i;j) = n+1-i + \#\{ a \in \ZZ : c_x = c_y = a \text{ for } x < i < j \leq y\}$  for $i \in [n]$ and $j \geq m$ is established.

Now we move on to the inductive step. Suppose that for all $i \in [n]$ and $j+1\geq i$, it holds that $(i; j+1) = n+1-i + \{a\in \ZZ : c_x = c_y = a \in \ZZ \text{ for } x < i < j+1 \leq y\}$. The backward inductive step on $j$ depending on whether $c_j = +, -$ or $c_j \in \ZZ$ in $c(\Lambda_{\bullet})$ is as follows.

\ben
\item[($\alpha$)] Suppose $c_j = +$. Then $(j;+) = \dim_{\K} ((\Lambda_j \cap V_+)/(\Lambda_{j+1} \cap V_+)) = 1$, and same arguments as in the proof of Lemma~\ref{affine-lmpli-overlap-dim-decr} show that there exist $v_j \in \Lambda_j \cap V_+$ such that $\Lambda_j = \Lambda_{j+1} \oplus \K v_j$. Therefore for any $i\in [n]$ and $ i\leq j$, it holds that
\begin{align*}
(i;j) & = \dim_{\K}(\pi_+(\Lambda_j) + \Lambda_i) \\
& = \dim_{\K}(\pi_+(\Lambda_{j+1} \oplus \K v_j) + \Lambda_i)\\
& = \dim_{\K}(\pi_+(\Lambda_{j+1}) + \K v_j + \Lambda_i)\\
& =  \dim_{\K}(\pi_+(\Lambda_{j+1}) + \Lambda_i) & (\text{since } v_j \in \Lambda_j \subset \Lambda_i)\\
& = (i;j+1).
\end{align*}
On the other hand since $c_j = +$, it holds that $\{a\in \ZZ : c_x = c_y = a \in \ZZ \text{ for } x < i < j \leq y\} = \{a\in \ZZ : c_x = c_y = a \in \ZZ \text{ for } x < i < j+1 \leq y\}$, and if $i = j$, it holds that $\{a\in \ZZ : c_x = c_y = a \in \ZZ \text{ for } x < i \leq y\} = \{a\in \ZZ : c_x = c_y = a \in \ZZ \text{ for } x < i < j+1 \leq y\}$. It follows that  
\begin{align*}
(i;j)&= (i; j+1)\\
& = n+1-i + \{a\in \ZZ : c_x = c_y = a \in \ZZ \text{ for } x < i < j+1 \leq y\}\\
& =  n+1-i + \{a\in \ZZ : c_x = c_y = a \in \ZZ \text{ for } x < i < j \leq y\}.
\end{align*}
\item[($\beta$)] Suppose $c_j = -$. Then by using the fact that $(j;-) = 1$ and similar steps as in case ($\alpha$), we have $(i;j) =  n+1-i + \{a\in \ZZ : c_x = c_y = a \in \ZZ \text{ for } x < i < j \leq y\}$ and $(i;i) = n+1-i + \{a\in \ZZ : c_x = c_y = a \in \ZZ \text{ for } x < i \leq y\}$.
\item[($\gamma$)] Suppose $c_j \in \ZZ$ and $c_j = c_{\ell}$ for some $\ell < j$. Notice that case (e) in Definition~\ref{alg-affine-flag-to-affine-clan} still applies to $c_j$ even though $j$ may not lie in $[n]$, since Lemma~\ref{affine-lmpli-overlap-dim-decr} and Lemma~\ref{affine-find-left-bracket} applies to all $j \in \ZZ$.
Therefore there exists a  $v_+ \in \pi_+(\Lambda_{j}) + \Lambda_{\ell+1}$, and we can write $v_+ = v_+' + v_+''$, where $v_+' \in \pi_+(\Lambda_j)$ and $v_+'' \in \Lambda_{\ell+1}$. The following facts concerning $v_+'$ hold:
\ben
\item[(1)] $v_+' \notin \pi_+(\Lambda_{j+1})$, since if $v_+' \in \pi_+(\Lambda_{j+1})$ then $v_+ =  v_+' + v_+'' \in \pi_+(\Lambda_{j+1}) + \Lambda_{\ell+1}$, which is a contradiction to the choice of $v_+$. Therefore $\pi_+(\Lambda_j) = \pi_+(\Lambda_{j+1}) \oplus \K v_+'$.
\item[(2)] $v_+' \notin \pi_+(\Lambda_{j+1}) + \Lambda_{i'}$ for any $i' > \ell$, since otherwise it holds that $v_+ = v_+' + v_+'' \in \pi_+(\Lambda_{j+1}) + \Lambda_{i'} + \Lambda_{\ell+1} = \pi_+(\Lambda_{j+1}) +  \Lambda_{\ell+1}$, contradicting the definition of $v_+$.
\item[(3)] $v_+' = v_+ - v_+'' \in \Lambda_{\ell}$ since $v_+ \in \Lambda_{\ell}$ and $v_+'' \in \Lambda_{\ell+1} \subset \Lambda_{\ell}$, and
\item[(4)] $v_+' \notin \Lambda_{i'}$ for any $i' >\ell $, since $v_+' \notin \Lambda_{\ell+1}$. In particular $v_+' \notin t\Lambda_1$.
\een
For $i$ satisfying $\min(j,n+1) >i >\ell$, we have
\begin{align*}
(i ; j ) &= \dim_{\K}( (\pi_+(\Lambda_j) + \Lambda_{i})/t\Lambda_1)\\
&= \dim_{\K} ((\pi_+(\Lambda_{j+1}) + \Lambda_{i} + \K v_+' ) /t\Lambda_1) & \text{by (1)}\\
& =  \dim_{\K} ((\pi_+(\Lambda_{j+1}) + \Lambda_{i}) /t\Lambda_1) + 1 &\text{by (2) and (4)}\\
& = ( i ; j+1) + 1.
\end{align*}
Observe in this case $\{a\in \ZZ : c_x = c_y = a \in \ZZ \text{ for } x < i < j \leq y\} = \{a\in \ZZ : c_x = c_y = a \in \ZZ \text{ for } x < i < j+1 \leq y\} + 1$. Therefore, we have
\begin{align*}
(i;j) &= (i ; j+1) + 1\\
& = n+1-i + \{a\in \ZZ : c_x = c_y = a \in \ZZ \text{ for } x < i < j+1 \leq y\} + 1\\
& = n+1-i + \{a\in \ZZ : c_x = c_y = a \in \ZZ \text{ for } x < i < j \leq y\}.
\end{align*}

On the other hand for $i \leq \ell$, we have
\begin{align*}
(i;j) &= \dim_{\K} ((\pi_+(\Lambda_j) + \Lambda_{i})/t\Lambda_1)\\
& = \dim_{\K}((\pi_+(\Lambda_{j+1}) + \K v_+' + \Lambda_{i})/t\Lambda_1)\\
& = \dim_{\K}((\pi_+(\Lambda_{j+1}) + \Lambda_{i})/t\Lambda_1) & \text{by (4) and $\Lambda_{\ell} \subset \Lambda_{i}$}\\
& = (i;j+1).
\end{align*}

But since $i \leq \ell$, it holds that $\{a\in \ZZ : c_x = c_y = a \in \ZZ \text{ for } x < i < j \leq y\} = \{a\in \ZZ : c_x = c_y = a \in \ZZ \text{ for } x < i < j+1 \leq y\}$. Therefore, we have
\begin{align*}
(i;j) &= (i ; j+1) \\
& = n+1-i + \{a\in \ZZ : c_x = c_y = a \in \ZZ \text{ for } x < i < j+1 \leq y\} \\
& = n+1-i + \{a\in \ZZ : c_x = c_y = a \in \ZZ \text{ for } x < i < j \leq y\}.
\end{align*}

\item[($\delta$)] Finally suppose $c_j \in \ZZ$ and $c_j = c_{\ell}$ for some $\ell > j$. The backward induction is straightforward, since we already know $(j;+) = (j;-) = 1$ and hence $\pi_+(\Lambda_j) = \pi_+(\Lambda_{j+1})$ by Lemma~\ref{affine-proj-dim}. By noticing that $\{a\in \ZZ : c_x = c_y = a \in \ZZ \text{ for } x < i < j \leq y\} = \{a\in \ZZ : c_x = c_y = a \in \ZZ \text{ for } x < i < j+1 \leq y\}$, we have 
\begin{align*}
(i;j) &=  \dim_{\K} ((\pi_+(\Lambda_j) + \Lambda_{i})/t\Lambda_1)\\
& =  \dim_{\K} ((\pi_+(\Lambda_{j+1}) + \Lambda_{i})/t\Lambda_1)\\
& = (i;j+1)\\
& = n+1-i + \{a\in \ZZ : c_x = c_y = a \in \ZZ \text{ for } x < i < j+1 \leq y\} \\
& = n+1-i + \{a\in \ZZ : c_x = c_y = a \in \ZZ \text{ for } x < i < j \leq y\}.
\end{align*}
\een

In the case $i=j$, replace $\{a\in \ZZ : c_x = c_y = a \in \ZZ \text{ for } x < i < j \leq y\}$ by $\{a\in \ZZ : c_x = c_y = a \in \ZZ \text{ for } x < i \leq y\}$ in all of the above cases and it follows that $(i;i) =  n+1-i + \{a\in \ZZ : c_x = c_y = a \in \ZZ \text{ for } x < i \leq y\}$.

By backward induction, we have shown that for the affine $(p,q)$-clan $c(\Lambda_{\bullet})$, it holds that $(i;j) = n+1-i + \#\{ a \in \ZZ : c_x = c_y = a \text{ for } x < i < j \leq y\}$, and $(i;i) = n+1-i + \#\{ a \in \ZZ : c_x = c_y = a \text{ for } x < i \leq y\}$ for $i \in [n]$ and $j > i$. This completes the proof.
\end{proof}

We now describe another algorithm which sends any affine $(p,q)$-clan to an affine flag, and hence showing that the map $\Lambda_\bullet \mapsto c(\Lambda_\bullet)$ is surjective.
\begin{proposition}\label{affine-clan-to-affine-flag}
Conversely, for any affine $(p,q)$-clan $c$, there exists an affine flag $\Lambda_{\bullet} = (\Lambda_1 \supset \Lambda_2 \dots \supset \Lambda_n)$ such that $c = c(\Lambda_{\bullet})$ is the affine $(p,q)$-clan of $\Lambda_{\bullet}$.
\end{proposition}

\begin{proof}
An affine flag is characterised by an ordered $A$-basis $\{v_1, v_2, \dots, v_n\}$ for $\Lambda_1$ such that for $1 \leq i \leq n$, it holds that 
\[\Lambda_i = \text{span}_{A} \{v_i, v_{i+1}, \dots, v_n, tv_1, tv_2 \dots, tv_{i-1}\}. \]
For any such ordered $A$-basis, there is a filtration of $\Lambda_1$:
\[ \{0\}  \subset \Lambda_{11} \subset \Lambda_{12} \subset \dots \subset \Lambda_{1(n-1)} \subset \Lambda_{1n}:= \Lambda_1, \]
with $\Lambda_{1j} = \text{span}_{A} \{v_1, v_2, \dots, v_j\}$ for $1\leq j \leq n$.

For any affine $(p,q)$-clan $c = ( \dots, c_1, c_2, \dots, c_n, \dots )$, an ordered $A$-basis $\{v_1, \dots, v_n\}$ for $\Lambda_1$ is given inductively as follows: for $1 \leq i \leq n$ we choose a non-zero $v_i \in F^n$ such that $\Lambda_{1i} = \text{span}_{A} \{v_1, v_2, \dots, v_i\}$.
\ben
\item[(1)] Suppose we have already obtained $v_1, v_2, \dots, v_{i-1}$, hence $\Lambda_{11}, \Lambda_{12}, \dots, \Lambda_{1(i-1)}$.
\ben
\item[(a)] If $c_i = +$, then set $v_i = e_s \in V_+$, where $s$ is the largest index between $1$ and $p$ such that $e_s \notin \pi_+(\Lambda_{1(i-1)})$.
\item[(b)] If $c_i = -$, then set $v_i = e_t \in V_-$, where $t$ is the largest index between $p+1$ and $n$ such that $e_t \notin \pi_-(\Lambda_{1(i-1)})$.
\item[(c)] If $c_i \equiv c_j \mod n$ for some $i<j \leq n $ and $c_i < c_j$, then set $v_i = e_s + t^{-m} e_t$, where $s$ and $t$ are as in $(a)$ and $(b)$, and $m = (c_j-c_i)/n$.
\item[(d)] If $c_i \equiv c_j \mod n$ for some $i<j \leq n $ and $c_i \geq c_j$, then set $v_i = e_t$, where $t$ is as in $(b)$.
\item[(e)] If $c_i \equiv c_j \mod n$ for some $1 \leq j<i$ and $c_i \leq c_j$, then set $v_i = e_s + t^{-m} e_t$, where $s$ and $t$ are the same as defined for $v_j$, and $m = (c_j-c_i)/n$.
\item[(f)] If $c_i \equiv c_j \mod n$ for some $1 \leq j <i $ and $c_i > c_j$, then set $v_i = e_t$, where $t$ is the same as defined for $v_j$.
\een
\item[(2)] Set $\Lambda_{1i} = \Lambda_{1(i-1)} \oplus A v_i$.
\een
We have constructed an ordered basis $\{v_1, v_2, \dots, v_n\}$ for $\Lambda_1$ such that by setting 
\[\Lambda_i = \text{span}_{A} \{v_i, v_{i+1}, \dots, v_n, tv_1, tv_2 \dots, tv_{i-1}\}\]
 for $i \in [n]$, it holds that $\Lambda_{\bullet} = (\Lambda_1 \supset \Lambda_2 \dots \supset \Lambda_n)$ is an affine flag, and $c(\Lambda_{\bullet}) = c$.
\end{proof}

Given an affine $(p,q)$-clan $c = ( \dots, c_1, c_2, \dots, c_n, \dots)$, define its associated matrix \defn{$M_{c}$}
to be a matrix in $\GL_n(F)$ such that $M_{c} = (v_n , v_{n-1} , \dots, v_2, v_1)$, where $v_i$'s are vectors defined in Proposition~\ref{affine-clan-to-affine-flag}. 

We now compare $M_c$ with the following notion of an \defn{affine $(p,q)$-clan matrix}. Suppose that $\sigma$ is an $n$-by-$n$ permutation matrix, and $M$ is a matching in $[n]$, consisting of ordered pairs $(i,j) \in [n] \times [n]$ with $i \neq j$ such that if $(i,j) \in M$, then $(i,k), (\ell, j) \notin M$ for all $k \neq j$ and $\ell \neq i$. Moreover, we choose $M$ satisfying the property that if $(i,j) \in M$ then $\sigma^{-1}(i) \leq p$ and $\sigma^{-1} (j) >p$.

\begin{definition}\label{def-affine-clan-matrix}
A matrix $g \in \GL_n(F)$ is called an affine $(p,q)$-clan matrix if it is a sum of 
\ben
\item[(i)] a permutation matrix $\sigma$ satisfying the property $\sigma^{-1}(1) < \sigma^{-1}(2) < \dots < \sigma^{-1}(p)$ and $\sigma^{-1}(p+1) < \sigma^{-1}(p+2) < \dots < \sigma^{-1}(p+q) = \sigma^{-1}(n)$, 
and
\item[(ii)] a matrix $L$ consisting of zeroes in all the entries except the following:
\ben
\item[(a)] if $(i,j) \in M$ and $i<j$, then the $(\sigma^{-1}(j), i)$ entry of $L$ is equal to $t^a$ for some $a \in \ZZ_{\leq 0}$,
\item[(b)] if $(i,j) \in M$ and $i>j$, then the  $(\sigma^{-1}(j), i)$ entry of $L$ is equal to $t^{a}$ for some $a \in \ZZ_{<0}$.
\een
\een
\end{definition}

\begin{remark}\label{affine-clan-to-matrix}
The map $c \mapsto \sigma_{c} M_{c}$ is a bijection between affine $(p,q)$-clans and affine $(p,q)$-clan matrices, where $\sigma_c$ is a suitable permutation matrix on the last $q$ rows of $M_c$.
\end{remark}

\begin{example}
For the affine $(3,3)$-clan $c$ with $(c_1, c_2, c_3, c_4, c_5, c_6) = (1, - , 2, + , -4, 13)$, it holds that 
\[ M_{c} = \begin{pmatrix}
0 &1 &0 &0 &0 &0\\
0 &0 &1 & 0& 0&0\\
 0&0 &0 & 0&0 &1\\
 0&t^{-1} &0 &1 &0 &0\\
0 &0 &0 &0 & 1&0\\
 1&0 &0 &0 & 0 &t^{-2} \end{pmatrix}. \]
Notice that multiplying the permutation $\sigma_c = (4, 5, 6)$ on the left of $M_c$ results in an affine $(p, q)$-clan matrix. Also, it is a straightforward exercise that $c(\Lambda_{\bullet}) = c$ by applying the algorithm described in Definition~\ref{alg-affine-flag-to-affine-clan}.
\end{example}

The following proposition is useful in proving that the map  $\Lambda_\bullet \mapsto c(\Lambda_\bullet)$ is an injection, which is necessary in the proof of Theorem~\ref{main-thm-lattice}.
\begin{proposition} \label{matrix-to-clan-direct-comp}
For any $g \in \GL_n(F)$, there exist $k \in K \cong \GL_p(F) \times \GL_q(F)$ and $b \in B$ such that $kgb$ is an affine $(p,q)$-clan matrix.
\end{proposition}
The proof involves direct calculations of $k_r \cdots k_2k_1gb_1b_2 \cdots b_s$ for suitable $k_1, k_2, \cdots, k_r \in K$ and $b_1, b_2, \cdots, b_s \in B$, and will be provided in the \hyperref[appendix]{appendix}.

Combining Propositions~\ref{affine-flag-to-affine-clan}, \ref{affine-clan-to-affine-flag} and \ref{matrix-to-clan-direct-comp}, we can prove Theorem~\ref{main-thm-lattice} now.

\begin{proof}[Proof of Theorem~\ref{main-thm-lattice}]
Combining Propositions~\ref{affine-flag-to-affine-clan} and \ref{affine-clan-to-affine-flag} alone proves that the map $\Lambda_\bullet \mapsto c(\Lambda_{\bullet})$ between the set of affine flags and affine $(p,q)$-clans is a surjection. By Lemma~\ref{k-inv-quantities}, this map is constant on each $K$-orbit. Therefore, it only remains to prove that  there do not exist two flags in different $K$-orbits but with the same set of data $(i;+), (i;-), (i;\NN)$ and $(i;j)$.

Proposition~\ref{affine-clan-to-affine-flag} gives a natural bijection between the set affine $(p,q)$-clans   and the set of affine $(p,q)$-clan matrices, as noted in Remark~\ref{affine-clan-to-matrix}. By Proposition~\ref{matrix-to-clan-direct-comp}, it holds that for each $g\in \GL_n(F)$, there exist $k \in K$ and $b \in B$ such that $kgb$ is an affine $(p,q)$-clan matrix. By Proposition~\ref{affine-clan-to-affine-flag}, the flags defined by each of these affine $(p,q)$-clan matrices have different sets of data $(i;+), (i;-), (i;\NN)$ and $(i;j)$. Therefore two flags with the same $K$-invariants must lie in the same $K$-orbit as the flag corresponding to their mutual affine $(p,q)$-clan, and hence the map $\Lambda_\bullet \mapsto c(\Lambda_\bullet)$ is also an injection.
\end{proof}

\appendix
\section{Proof of Proposition~\ref{matrix-to-clan-direct-comp} }\label{appendix}
We begin by proving a lemma that simplifies the proof.
Define a \defn{partial identity matrix} to be a matrix which becomes an identity matrix after deleting the all‑zero columns.
Examples of partial identity matrices are
\[
\begin{pmatrix}
1 & 0 & 0 & 0 & 0 & 0 & 0\\
0 & 0 & 1 & 0 & 0 & 0 & 0\\
0 & 0 & 0 & 0 & 1 & 0 & 0\\
0 & 0 & 0 & 0 & 0 & 0 & 1
\end{pmatrix} \quand
\begin{pmatrix}
0 & 1& 0 & 0 & 0 & 0 & 0\\
0 & 0 & 1 & 0 & 0 & 0 & 0\\
0 & 0 & 0 & 0 & 1 & 0 & 0\\
0 & 0 & 0 & 0 & 0 & 1 & 0
\end{pmatrix}.
\]

For $i, j \in [n]$, let $E_{ij}$ denote the $n$-by-$n$ matrix whose $(i,j)$ entry is the only non-zero entry and is equal to $1$. For $m \in \ZZ$, we define $U_{ij,m} := I_n + t^mA E_{ij}$.  Let $T(A)$ be the multiplicative group of $n$-by-$n$ diagonal matrices with all entries in $A^*$. Then the Iwahori subgroup $B$ (upper triangular modulo $t$) of $\GL_n(A)$ can be written as  \cite[Theorem 2.5]{IM}:
\[ B = {T}(A) \prod_{1\leq i<j\leq n}U_{ij,0} \prod_{n\geq i>j \geq 1} U_{ij,1}.\]

Therefore, right multiplication by a matrix $b\in B$ to a matrix $g \in G$ performs a sequence of the following operations: 
\ben
\item[(i)] add $A$-multiples of previous columns of $g$ to the latter columns of $g$, 
\item[(ii)] add $tA$-multiples of latter columns of $g$ to previous columns of $g$, and 
\item[(iii)] multiply an element in $A^*$ to a particular column in $g$. 
\een

\begin{lemma} \label{large-block-reduction}
Suppose $m \leq n$. For each full-rank $m$-by-$n$ matrix $M$ with entries in $F$, there exist elements $g$ in $\GL_m(F)$ and $h$ in $B$ such that $gMh$ is a partial identity matrix.

\begin{proof}
From the first row of any matrix $M$, pick the leftmost entry of minimal order, say in column $a_1$, and use the column operations by right multiplication of $B$ to eliminate all other non-zero entries in the first row. Then from the second row, pick the leftmost entry of minimal order other than the entry in $(2, a_1)$, say in column $a_2$, and use the column operations by right multiplication of $B$ to eliminate all other non-zero entries in the second row except entries in $(2,a_1)$ and $(2,a_2)$. 

Inductively, after right multiplication by elements in $B$, the only non-zero entries of the $m$-by-$n$ matrix will be located at columns $a_1, a_2, \dots a_m$ inductively defined, and the $m$-by-$m$ matrix formed by putting columns $a_1,a_2, \dots a_m$ together is invertible, since the rows are linearly independent. 
Denote the resulting $m$-by-$n$ matrix as $Y = ({y_1}, {y_2}, \dots, {y_n})$ with ${y_i}$ being the $i$-th column of $Y$. 
Define $\sigma \in S_m$ such that $a_{\sigma(1)} < a_{\sigma(2)} < \dots < a_{\sigma(m)}$.  
Now left multiplication by the inverse of $({y_{a_{\sigma(1)}}}, {y_{a_{\sigma(2)}}}, \dots, {y_{a_{\sigma(m)}}})$ gives an $m$-by-$n$ matrix with columns $a_{\sigma(1)}, a_{\sigma(2)}, \dots , a_{\sigma(m)}$ forming the identity matrix, while all other columns are zero.
\end{proof}
\end{lemma}

\begin{example}
Suppose $M = \begin{pmatrix} 1 & t^{-1} \end{pmatrix}$. Then by taking $\begin{pmatrix} t \end{pmatrix} \in \GL_1(F)$ and $\begin{pmatrix} 1 & 0\\ -t & 1 \end{pmatrix} \in B \subset \GL_2(F)$, we have
\[ \begin{pmatrix} t \end{pmatrix} M \begin{pmatrix} 1 & 0\\ -t & 1 \end{pmatrix} = \begin{pmatrix} t \end{pmatrix} \begin{pmatrix} 0 & t^{-1} \end{pmatrix} = \begin{pmatrix} 0 & 1 \end{pmatrix}, \]
which is a partial identity matrix.
\end{example}
Before we consider the general case, let us look at the special case when $p=q=1$, and classify the orbit representatives in $KgB$ with $K = \GL_1(F) \times \GL_1(F)$ and $B \subset \GL_2(F)$ the Iwahori subgroup (upper triangular modulo $t$). 

\begin{lemma} \label{11affineclan}
In the case $n=2$, $p = q = 1$, for each $g \in \GL_2(F)$, there exist $k \in K$ and $b \in B$ such that $kgb$ is equal to either $\begin{pmatrix} 1 & 0 \\ 0 & 1 \end{pmatrix}$, $\begin{pmatrix} 0 & 1 \\ 1 & 0 \end{pmatrix}$, $\begin{pmatrix} 1 & 0 \\ t^a & 1 \end{pmatrix}$ with $a \leq 0$ or $\begin{pmatrix} 0 & 1\\ 1 & t^b \end{pmatrix}$ with $b <0$. 
\end{lemma}

\begin{proof}
Using Lemma~\ref{large-block-reduction} with $m=1$ and $n=2$, each double coset contains a matrix whose first row is either $\begin{pmatrix} 1 & 0 \end{pmatrix}$ or $\begin{pmatrix} 0 & 1 \end{pmatrix}$ under the right action of $B$ and left action of multiplication of non-zero formal Laurent series. We consider these two cases separately.

\ben
\item[(a)] First suppose a double coset contains a matrix of the form $g = \begin{psmallmatrix} 1 & 0 \\ x & y \end{psmallmatrix}$. Since this matrix must be invertible, it holds that $y \neq 0$. 
\ben
\item[(i)] Suppose $x = 0$. In this case we have $\begin{psmallmatrix} 1 & 0 \\ 0 & y^{-1} \end{psmallmatrix} \in K $ and $\begin{psmallmatrix} 1 & 0 \\ 0 & y^{-1} \end{psmallmatrix} g = \begin{psmallmatrix} 1 & 0 \\ 0 & 1 \end{psmallmatrix}$.
\item[(ii)] Suppose $x \neq 0$. If $\ord (x) > \ord (y)$, then $\ord  (x/y) >0$, and hence $\begin{psmallmatrix} 1 & 0 \\ -x/y & 1 \end{psmallmatrix} \in B$. Observe that  
\[\begin{psmallmatrix} 1 & 0 \\ 0 & y^{-1} \end{psmallmatrix} \begin{psmallmatrix} 1 & 0 \\ x & y \end{psmallmatrix}\begin{psmallmatrix} 1 & 0 \\ -x/y & 1 \end{psmallmatrix} =  \begin{psmallmatrix} 1 & 0 \\ 0 & y^{-1} \end{psmallmatrix} \begin{psmallmatrix} 1 & 0 \\ 0 & y \end{psmallmatrix} = \begin{psmallmatrix} 1 & 0 \\ 0 & 1 \end{psmallmatrix}. \] 
If $\ord x \leq \ord y$, then $ \ord (x/y) \leq 0$, and hence $x/y = t^a u$ for some $a \leq 0$ and $u \in A^*$. Observe that
\[ \begin{psmallmatrix} u & 0 \\ 0 & y^{-1} \end{psmallmatrix} \begin{psmallmatrix} 1 & 0 \\ x & y \end{psmallmatrix}\begin{psmallmatrix} u^{-1} & 0 \\ 0 & 1 \end{psmallmatrix} =  \begin{psmallmatrix} u & 0 \\ t^a u & 1 \end{psmallmatrix} \begin{psmallmatrix} u^{-1} & 0 \\ 0 & 1 \end{psmallmatrix} = \begin{psmallmatrix} 1 & 0 \\ t^a & 1 \end{psmallmatrix}. \] 
\een
\item[(b)] Now suppose a double coset contains a matrix of the form $ h = \begin{psmallmatrix} 0 & 1 \\ x & y \end{psmallmatrix}$. Since this matrix must be invertible, it holds that $x \neq 0$. 
\ben
\item[(i)] Suppose $y = 0$. In this case we have $\begin{psmallmatrix} 1 & 0 \\ 0 & x^{-1} \end{psmallmatrix} \in K$ and $\begin{psmallmatrix} 1 & 0 \\ 0 & x^{-1} \end{psmallmatrix}h =  \begin{psmallmatrix} 0 & 1 \\ 1 & 0 \end{psmallmatrix}$.
\item[(ii)] Suppose $y \neq 0$. If $\ord (y) \geq \ord (x)$, then $\ord  (y/x) \geq 0$. Therefore $\begin{psmallmatrix} 1 & -y/x \\ 0 & 1 \end{psmallmatrix} \in B$, and 
\[\begin{psmallmatrix} 1 & 0 \\ 0 & x^{-1} \end{psmallmatrix} \begin{psmallmatrix} 0 & 1 \\ x & y \end{psmallmatrix}\begin{psmallmatrix} 1 & -y/x \\ 0 & 1 \end{psmallmatrix} =  \begin{psmallmatrix} 1 & 0 \\ 0 & x^{-1} \end{psmallmatrix} \begin{psmallmatrix} 0 & 1 \\ x & 0 \end{psmallmatrix} = \begin{psmallmatrix} 0& 1 \\ 1 & 0 \end{psmallmatrix}. \] 
If $\ord (y) < \ord (x)$, then $ \ord (y/x) < 0$, and hence $y/x = t^b u$ for some $b < 0$ and $u \in A^*$. Observe that
\[ \begin{psmallmatrix} u & 0 \\ 0 & x^{-1} \end{psmallmatrix} \begin{psmallmatrix} 0 & 1 \\ x & y \end{psmallmatrix}\begin{psmallmatrix} 1 & 0 \\ 0 & u^{-1} \end{psmallmatrix} =  \begin{psmallmatrix} 0 & u \\ 1 & t^bu \end{psmallmatrix} \begin{psmallmatrix} 1 & 0 \\ 0 & u^{-1} \end{psmallmatrix} = \begin{psmallmatrix} 0 & 1 \\  1& t^b \end{psmallmatrix}. \] 
\een
\een
Thus we know that an exhaustive list of double coset representatives are $\begin{psmallmatrix} 1 & 0 \\ 0 & 1 \end{psmallmatrix}$, $\begin{psmallmatrix} 0 & 1 \\ 1 & 0 \end{psmallmatrix}$, $\begin{psmallmatrix} 1 & 0 \\ t^a & 1 \end{psmallmatrix}$ with $a \leq 0$, and $\begin{psmallmatrix} 0 & 1\\ 1 & t^b \end{psmallmatrix}$ with $b <0$. 
\end{proof}

We are ready to prove Proposition~\ref{matrix-to-clan-direct-comp}.

\begin{proof}[Proof of Proposition~\ref{matrix-to-clan-direct-comp}]
We provide an explicit algorithm consisting of a series of left and right matrix multiplications in $K$ and $B$ respectively, which transforms each matrix $g \in \GL_n(F)$ into an affine $(p,q)$-clan matrix. Write any element $k \in K$ in the form $k = \begin{pmatrix} k_{1} & 0 \\ 0 & k_{2} \end{pmatrix}$ with $k_{1} \in \GL_p(F)$ and $k_{2} \in \GL_q(F)$. 

Observe that left multiplication of a matrix in $K$ performs row operations independently on the first $p$ rows and the last $q$ rows. Specifically, if a set of $p$ columns within the first $p$ rows forms an invertible $p$-by-$p$ matrix $P$, then multiplying a matrix $k$ on the left, where $k_{1} = P^{-1}$ and $k_{2} = I_q$, would transform the $p$-by-$p$ submatrix  into an identity matrix. This operation may modify entries in the remaining $q$ columns of the first $p$ rows, although these columns are mostly filled with zeros in the subsequent calculations. 

Similarly, if a set of $q$ columns within the last $q$ rows forms an invertible $q$-by-$q$ matrix $Q$, then multiplying matrix $k$ on the left, where $k_{1} = I_p$ and $k_{2}= Q^{-1}$, would transform the $q$-by-$q$ submatrix into an identity matrix.

For example in the case concerning certain $p$-columns of the first $p$-rows,
\begin{equation}\label{reduce-to-identity-first-p-rows}
\resizebox{\textwidth}{!}{%
$\begin{pmatrix}
P^{-1} & 0_{p \times q}\\
0_{q \times p} & I_q
\end{pmatrix}
\begin{pmatrix}
& * & \dots & * & & &  & * &   \\
& * &  \dots   & * & & & & * &  \\
& * &  \dots    & * & & & & * &   \\
v_1 & \vdots & & \vdots & v_2& v_3 & \dots & \vdots & v_p \\
& * & \dots & * & & &  & * &  \\
& * &         & * & & & & * & \\
\hline
\# & \# &\dots &\#&& \dots &&\#&\#\\
\vdots & \vdots & \ddots& \vdots & & \ddots & & \vdots&\vdots \\
\# & \# &\dots  & \#&& \dots &&\#&\#
\end{pmatrix}
=
\begin{pmatrix}
1& * & \dots & * & 0& 0&  & * &0  \\
 0 & * & \dots  & * & 1& 0& & * &0 \\
 & * &  \dots  & * & 0 & 1& & * & 0 \\
\vdots & \vdots & & \vdots & \vdots& \vdots & \dots & \vdots & \vdots \\
0& * & \dots & * &0 & 0&  & * & 0\\
0& * &         & * & 0& 0& & * &1 \\
\hline
\# & \# &\dots &\#&& \dots &&\#&\#\\
\vdots & \vdots & \ddots& \vdots & & \ddots & & \vdots&\vdots \\
\# & \# &\dots  & \#&& \dots &&\#&\#
\end{pmatrix}$.}
\end{equation}
Here $v_1, v_2, \dots, v_p \in F^p$, and $P = (v_1\, v_2\, \dots\, v_p)$.
The entries with $*$ might change accordingly, while entries with $\#$ in last $q$ rows will not change. Similarly left multiplication by $\begin{pmatrix}
I_p & 0_{p \times q}\\
0_{q \times p} & Q^{-1}
\end{pmatrix}$
transform a certain $q$-by-$q$ submatrix in the last $q$ rows into an identity matrix.

With the tools above we start to reduce each matrix in $\GL_n(F)$ to an affine $(p,q)$-clan matrix:

Step $1$:
Using Lemma~\ref{large-block-reduction}, we can reduce the bottom $q$ rows into a $q$ by $n$ matrix such that $p$ columns are zero, and the remaining $q$ columns form a $q$ by $q$ identity matrix. Say the $q$ columns are the columns $i_1< i_2< \dots < i_q$. 
After the reduction, the matrix is in this form:
\[
\begin{pmatrix}
*& * & * &\dots &* & *& \dots &     *  \\
 * & * &* & \dots &* &*&  \dots &    * \\
 \vdots& \vdots  &   \vdots &\ddots & \vdots&  \vdots & \ddots&   \vdots\\
  *& * &  *&\dots & * & * & \dots&    *  \\
    \hline
0 & 1 & 0 &\dots &0&  0& \dots&    0\\
0 & 0 & 0 & \dots &0& 1&\dots&0\\
\vdots & \vdots & \vdots & \ddots &\vdots &  \vdots&\ddots& \vdots \\
0 & 0 & 0 & \dots &0& 0&\dots&1
\end{pmatrix}.
\]
Denote the matrix in the above form by $g_1$.

Step $2$: Consider the leftmost smallest order entry for the first row of $g_1$, which is located in the entry $(1,j_1)$.
\ben
\item[(a)] Suppose $j_1 \notin \{i_1, i_2, \dots, i_q\}$. Then the last $q$ entries in column $j_1$ are all zeroes, and hence in the process of eliminating all the other non-zero entries in the first row using restricted column operations described in the paragraph before step $1$, all the entries in the last $q$ rows of the matrix $g_1$ do not change.

This case is illustrated as follows. Suppose the red $*$ in the first row is the $j_1$'th column such that $j_1 \notin \{i_1, i_2, \dots, i_q\}$. All the non-zero entries in the first row other than the red $*$ are eliminated by the entry in red $*$ using column operations.
\[
\begin{pmatrix}
*& * & \textcolor{red}{*} &\dots &* & *& \dots      &*  \\
 * & * &* & \dots &* &*&  \dots     &* \\
 \vdots&\vdots   &  \vdots  &\ddots &\vdots & \vdots  & \ddots   & \vdots\\
  *& * &  *&\dots & * & * & \dots    &*  \\
  \hline
0 & 1 & 0 &\dots &0&  0& \dots  &  0\\
0 & 0 & 0 & \dots &0& 1&\dots&0\\
\vdots & \vdots & \vdots & \ddots &\vdots &  \vdots&\ddots  & \vdots \\
0 & 0 & 0 & \dots &0& 0&\dots&1
\end{pmatrix}
\xrightarrow{\text{column operations}}
\begin{pmatrix}
0& 0 & \textcolor{red}{*} &\dots &0 & 0& \dots     &0 \\
 * & * &* & \dots &* &*&  \dots     &* \\
 \vdots& \vdots  &  \vdots  &\ddots & \vdots&   \vdots& \ddots   & \vdots\\
  *& * &  *&\dots & * & * & \dots    &*  \\
  \hline
0 & 1 & 0 &\dots &0&  0& \dots  &  0\\
0 & 0 & 0 & \dots &0& 1&\dots &0\\
\vdots & \vdots & \vdots & \ddots &\vdots &  \vdots&\ddots&   \vdots \\
0 & 0 & 0 & \dots &0& 0&\dots&1
\end{pmatrix}.
\]

\item[(b)] Suppose $j_1 = i_a$ for some $a \in [1,q]$.  Then by eliminating all non-zero entries in the first row of $g_1$ by the entry $(1, j_1)$ using restricted column operations, some entries in row $p+a$ become non-zero again. If any of the entries $(p+a, i_b)$ is non-zero for $a \neq b \in [1,q]$, use row operation on the last $q$ rows to eliminate the non-zero entries $(p+a, i_b)$ by the pivot $1$ in $(p+b ,i_b)$.

This case is illustrated as follows. Suppose the red $*$ in the first row of the $j_1$'th column is equal to $i_1$. Notice that after the first arrow all the non-zero entries in the first row other than the red $*$ are eliminated by the entry in red $*$ using column operations, but some entries in column $p+1$ becomes non-zero in this process. After the second arrow, some non-zero entries in row $p+1$ are eliminated by the pivot columns in the last $q$ rows of the matrix after the first arrow.
\begin{align*}
\begin{pmatrix}
*& \textcolor{red}{*} & *  &\dots &* & *& \dots      &*  \\
 * & * &* & \dots &* &*&  \dots     &* \\
 \vdots& \vdots  &  \vdots  &\ddots & \vdots&  \vdots & \ddots  & \vdots\\
  *& * &  *&\dots & * & * & \dots    &*  \\
    \hline
0 & 1 & 0 &\dots &0&  0& \dots  &  0\\
0 & 0 & 0 & \dots &0& 1&\dots &0\\
\vdots & \vdots & \vdots & \ddots &\vdots &  \vdots&\ddots&   \vdots \\
0 & 0 & 0 & \dots &0& 0&\dots&1
\end{pmatrix}
 \xrightarrow{\text{ \quad \quad \quad    column operations  \quad\quad\quad       }}&
\begin{pmatrix}
0& \textcolor{red}{*} & 0&\dots &0 & 0& \dots      &0 \\
 * & * &* & \dots &* &*&  \dots    &* \\
 \vdots& \vdots  &  \vdots  &\ddots &\vdots &  \vdots & \ddots  & \vdots\\
  *& * &  *&\dots & * & * & \dots    &*  \\
    \hline
* & 1 & * &\dots &*&  *& \dots &  *\\
0 & 0 & 0 & \dots &0& 1&\dots&0\\
\vdots & \vdots & \vdots & \ddots &\vdots &  \vdots&\ddots&  \vdots \\
0 & 0 & 0 & \dots &0& 0&\dots&1
\end{pmatrix}\\
 \xrightarrow{\text{row operations on the pivot columns}}&
\begin{pmatrix}
0& \textcolor{red}{*} & 0&\dots &0 & 0& \dots      &0 \\
 * & * &* & \dots &* &*&  \dots    &* \\
 \vdots&  \vdots &  \vdots  &\ddots &\vdots & \vdots  & \ddots   & \vdots\\
  *& * &  *&\dots & * & * & \dots    &*  \\
    \hline
* & 1 & * &\dots &*&  0& \dots  &  0\\
0 & 0 & 0 & \dots &0& 1&\dots&0\\
\vdots & \vdots & \vdots & \ddots &\vdots &  \vdots&\ddots  & \vdots \\
0 & 0 & 0 & \dots &0& 0&\dots&1
\end{pmatrix}.
\end{align*}
\een

Denote the resulting matrix after this step by $g_2$.

Step $3$: After step $2$, do row operations to eliminate the non-zero entries below the entry $(1, j_1)$ in the first $p$ rows of $g_2$, and multiply the multiplicative inverse of the entry $(1,j_1)$ to the first row of $g_2$ to reduce the entry $(1,j_1)$ to $1$. These row operations correspond to multiplying elements in $K$ on the left with $g_2$. Then the resulting matrix for the case with $j_1 \notin \{i_1, i_2, \dots, i_q\}$ in Step $2$(a) is equal to
\[
\begin{pmatrix}
0& 0 & 1 &\dots &0 & 0& \dots     &0 \\
 * & * &0 & \dots &* &*&  \dots     &* \\
 \vdots&  \vdots &  \vdots  &\ddots & \vdots&  \vdots & \ddots   & \vdots\\
  *& * &  0&\dots & * & * & \dots    &*  \\
    \hline
0 & 1 & 0 &\dots &0&  0& \dots  &  0\\
0 & 0 & 0 & \dots &0& 1&\dots&0\\
\vdots & \vdots & \vdots & \ddots &\vdots &  \vdots&\ddots&   \vdots \\
0 & 0 & 0 & \dots &0& 0&\dots&1
\end{pmatrix},
\]
while the resulting matrix for the case with $j_1 = i_a$ for some $a \in [1,q]$ in Step $2$(b) is equal to
\[
\begin{pmatrix}
0& 1& 0&\dots &0 & 0& \dots      &0 \\
 * & 0 &* & \dots &* &*&  \dots    &* \\
 \vdots&  \vdots &  \vdots  &\ddots & \vdots& \vdots& \ddots   & \vdots\\
  *& 0 &  *&\dots & * & * & \dots    &*  \\
    \hline
* & 1 & * &\dots &*&  0& \dots  &  0\\
0 & 0 & 0 & \dots &0& 1&\dots&0\\
\vdots & \vdots & \vdots & \ddots &\vdots &  \vdots&\ddots  & \vdots \\
0 & 0 & 0 & \dots &0& 0&\dots&1
\end{pmatrix}.
\]
Denote the resulting matrix after this step by $g_3$.

Step $4$: Move on to consider steps $2$ and $3$ for rows $2$ to $p$, with ``second row" replacing the ``first row". 
Inductively we obtain a matrix $g_4' \in \GL_n(F)$ satisfying 
\ben
\item[(i)] columns $j_1, j_2, \dots j_p$ in the first $p$ rows form a permutation matrix and all other columns having $0$ entries in the first $p$ coordinates, 
\item[(ii)] columns $i_1, i_2, \dots i_q$ in the last $q$ rows form an identity matrix, and 
\item[(iii)] the entries $(p+ \beta, \alpha)$, where $\alpha \notin \{i_1, i_2, \dots i_q\}$ and $\beta \in [1,q]$, could be non-zero if and only if $i_{\beta} \in \{j_1, j_2, \dots, j_p\}$.
\een
By multiplying $g_4'$ by a matrix in $K$ on the left such that $k_{1}$ permutes the first $p$ rows, we can assume that $j_1<j_2 < \dots <j_p$, without loss of generality.

An example of the present situation where $p = q =4$ and $j_1=1$, $j_2= 3$, $j_3 = 4$, $j_4= 8$, and $i_1=2$, $i_2= 3$, $i_3 = 7$, $i_4 = 8$ is as follows:

\begin{equation}\label{ex-few-nonzero-entries-but-still-can-be-fewer}
\begin{pmatrix}
1 & 0 & 0 & 0 & 0 & 0 & 0 & 0\\
0 & 0 & 1 & 0 & 0 & 0 & 0 & 0\\
0 & 0 & 0 & 1 & 0 & 0 & 0 & 0\\
0 & 0 & 0 & 0 & 0 & 0 & 0 & 1\\
  \hline
0 & 1 & 0 & 0 & 0 & 0 & 0 & 0\\
\textcolor{red}{*} & 0 & 1 & \textcolor{red}{*} & * & * & 0 & 0\\
0 & 0 & 0 & 0 & 0 & 0 & 1 & 0\\
\textcolor{red}{*}  & 0 & 0 & \textcolor{red}{*} & * & * & 0 & 1\\
\end{pmatrix},
\end{equation}

while an example with the same $p = q=4$ but $i_1=j_1 = 2$, $i_2 = j_2 = 3$, $i_3 = j_3 = 4$ and $i_4= j_4 = 5$ is as follows:
\begin{equation}
\begin{pmatrix}\label{ex-max-nonzero-entries}
0 & 1 & 0 & 0 & 0 & 0 & 0 & 0 \\
0 & 0 & 1 & 0 & 0 & 0 & 0 & 0 \\
0 & 0 & 0 & 1 & 0 & 0 & 0 & 0 \\
0 & 0 & 0 & 0 & 1& 0 & 0 & 0 \\
  \hline
* & 1 & 0 & 0 & 0 & * & * & * \\
* & 0 & 1 & 0 & 0 & * & * & * \\
* & 0 & 0 & 1 & 0 & * & * & * \\
* & 0 & 0 & 0 & 1 & * & * & * \\
\end{pmatrix}.
\end{equation}
Here $*$ are any elements in $A$ by construction, as long as the matrix is in $\GL_n(F)$. Moreover the entries $(p+ \beta , \alpha)$ are in $tA$ if $\alpha < i_{\beta}$.

Denote the resulting matrix after this step by $g_4$.

Step $5$: If all the $(p+\beta, \alpha)$ entries of $g_4$ are zero, where $\alpha \in \{j_1, j_2, \dots, j_p\}\setminus \{i_1, i_2, \dots, i_q\}$ and $\beta \in [1,q]$, then we skip this step. Suppose that some $(p+\beta, \alpha)$ entries of $g_4$ are non-zero. We can eliminate these entries as follows. Since these entries are either in $tA$ if $\alpha < i_{\beta}$ or in $A$ if $\alpha>i_{\beta}$, we can eliminate the $(p+\beta, \alpha)$ entries using the pivot $1$ in the column $i_{\beta}$ by column operations. 

Now there are some non-zero entries in the first $p$ entries of the columns $i_{\beta}$ in the resulting matrix. However, using the formula \eqref{reduce-to-identity-first-p-rows}, we can offset the extra non-zero entries in the first $p$ rows of the resulting matrix.

The red entries in example \eqref{ex-few-nonzero-entries-but-still-can-be-fewer} are eliminated in the following way:
\begin{align*}
\begin{pmatrix}
1 & 0 & 0 & 0 & 0 & 0 & 0 & 0\\
0 & 0 & 1 & 0 & 0 & 0 & 0 & 0\\
0 & 0 & 0 & 1 & 0 & 0 & 0 & 0\\
0 & 0 & 0 & 0 & 0 & 0 & 0 & 1\\
  \hline
0 & 1 & 0 & 0 & 0 & 0 & 0 & 0\\
\textcolor{red}{*} & 0 & 1 & \textcolor{red}{*} & * & * & 0 & 0\\
0 & 0 & 0 & 0 & 0 & 0 & 1 & 0\\
\textcolor{red}{*}  & 0 & 0 & \textcolor{red}{*} & * & * & 0 & 1\\
\end{pmatrix}
\xrightarrow{\text{column operations}}
& \begin{pmatrix}
1 & 0 & 0 & 0 & 0 & 0 & 0 & 0\\
* & 0 & 1 & * & 0 & 0 & 0 & 0\\
0 & 0 & 0 & 1 & 0 & 0 & 0 & 0\\
* & 0 & 0 & * & 0 & 0 & 0 & 1\\
  \hline
0 & 1 & 0 & 0 & 0 & 0 & 0 & 0\\
0 & 0 & 1 & 0 & * & * & 0 & 0\\
0 & 0 & 0 & 0 & 0 & 0 & 1 & 0\\
0 & 0 & 0 & 0 & * & * & 0 & 1\\
\end{pmatrix}\\
\xrightarrow{\quad\quad \eqref{reduce-to-identity-first-p-rows}\quad\quad}
&\begin{pmatrix}
1 & 0 & 0 & 0 & 0 & 0 & 0 & 0\\
0 & 0 & 1 & 0 & 0 & 0 & 0 & 0\\
0 & 0 & 0 & 1 & 0 & 0 & 0 & 0\\
0 & 0 & 0 & 0 & 0 & 0 & 0 & 1\\
  \hline
0 & 1 & 0 & 0 & 0 & 0 & 0 & 0\\
0 & 0 & 1 & 0 & * & * & 0 & 0\\
0 & 0 & 0 & 0 & 0 & 0 & 1 & 0\\
0 & 0 & 0 & 0 & * & * & 0 & 1\\
\end{pmatrix}.
\end{align*}
In other words, condition (i) and (ii) in Step $4$ still hold, but instead of (iii) the resulting matrix has the following property:
\ben
\item[(iii$'$)] the entries $(p+ \beta, \alpha)$ with $\beta \in [1,q]$ could be non-zero if and only if $i_{\beta} \in \{j_1, j_2, \dots, j_p\}$ and $\alpha \notin \{i_1, i_2, \dots, i_q, j_1, j_2, \dots, j_p\}$.
\een
Denote the resulting matrix after this step by $g_5$.

Step $6$: Now suppose $r$ of the columns $i_1< i_2<  \dots < i_q$ and $j_1<j_2 < \dots <j_p$ coincide, i.e. there exists $a_1< a_2< \dots < a_r \in [1, p]$ and $b_1< b_2< \dots < b_r \in [1,q]$ such that $i_{a_s} = j_{b_s}$ for all $s \in [1,r]$. 
We reduce the resulting matrix $g_5$ to find out the matchings of an affine $(p,q)$-clan matrix by the following inductive procedure. 
\ben
\item[(a)]Among the only non-zero entries in the columns $\alpha \in [n]\setminus \{i_1, \dots i_q, j_1, \dots j_p\}$ of $g_5$, choose a bottom-leftmost smallest order entry, and suppose it is located in $(p+\beta, \alpha)$ with $\beta \in [1,q]$. 
Eliminate any non-zero entries in the last $q$ entries of column $\alpha$ by row operations on the last $q$ rows (corresponding to left multiplications by elements in $K$), and eliminate any non-zero entries other than pivot $1$'s in row $p+\beta$ by column operations (corresponding to right multiplications by elements in $B$), which is possible by the choice of $(p+\beta, \alpha)$.
\item[(b)]However in the pivot column $i_{\beta}$ of the resulting matrix after step $6$(a), some entries in the last $q$ coordinates may be non-zero after the row operations, but with non-negative order if the entry is above $(p+\beta, i_{\beta})$, and positive order of the entry is below $(p+\beta, i_{\beta})$. Eliminate these entries by column operations with the pivot columns $i_1, i_2, \dots, i_{c-1}, i_{c+1}, \dots, i_q$, which correspond to right multiplications by elements in $B$.
\item[(c)] Then use the identity \eqref{reduce-to-identity-first-p-rows} to eliminate any extra non-zero entries in the first $p$ rows. 
\item[(d)]Notice that $i_{\beta} = j_{b_{\ell}}$ for some $\ell \in [1,r]$.
Therefore the only non-zero elements in the columns $\alpha$ and $i_{\beta}$ of the resulting matrix after step $6$(c) are the entry in $(p+\beta, \alpha)$ and the pivot $1$'s in column $i_{\beta} = j_{b_{\ell}}$. and they are the only non-zero elements in the rows $b$ and $p+\beta$, and hence the columns $\alpha$ and $i_{\beta}$ are matched in constructing the incomplete matching $M$ of the affine $(p,q)$-clan matrix. 
\een

To illustrate the above procedures, suppose in the example \eqref{ex-max-nonzero-entries}, the $(4+3, 6)$-entry is the bottom-leftmost smallest order entry, labeled red below:
\[
\begin{pmatrix}
0 & 1 & 0 & 0 & 0 & 0 & 0 & 0 \\
0 & 0 & 1 & 0 & 0 & 0 & 0 & 0 \\
0 & 0 & 0 & 1 & 0 & 0 & 0 & 0 \\
0 & 0 & 0 & 0 & 1& 0 & 0 & 0 \\
  \hline
* & 1 & 0 & 0 & 0 & a_{56} & * & * \\
* & 0 & 1 & 0 & 0 & a_{66} & * & * \\
a_{71} & 0 & 0 & 1 & 0 & \textcolor{red}{a_{76}} & a_{77} & a_{78} \\
* & 0 & 0 & 0 & 1 & a_{86} & * & * \\
\end{pmatrix}.
\]

By definition, we have $\ord(a_{56}), \ord(a_{66}), \ord(a_{77}), \ord(a_{78}) \geq \ord(a_{76})$ and \newline
$\ord(a_{71}), \ord(a_{86}) > \ord(a_{76})$. The procedures above will act as follows:

\begin{align*}
\begin{pmatrix}
0 & 1 & 0 & 0 & 0 & 0 & 0 & 0 \\
0 & 0 & 1 & 0 & 0 & 0 & 0 & 0 \\
0 & 0 & 0 & 1 & 0 & 0 & 0 & 0 \\
0 & 0 & 0 & 0 & 1& 0 & 0 & 0 \\
  \hline
* & 1 & 0 & 0 & 0 & a_{56} & * & * \\
* & 0 & 1 & 0 & 0 & a_{66} & * & * \\
a_{71} & 0 & 0 & 1 & 0 & \textcolor{red}{a_{76}} & a_{77} & a_{78} \\
* & 0 & 0 & 0 & 1 & a_{86} & * & * \\
\end{pmatrix}
&\xrightarrow{(a)}
\begin{pmatrix}
0 & 1 & 0 & 0 & 0 & 0 & 0 & 0 \\
0 & 0 & 1 & 0 & 0 & 0 & 0 & 0 \\
0 & 0 & 0 & 1 & 0 & 0 & 0 & 0 \\
0 & 0 & 0 & 0 & 1& 0 & 0 & 0 \\
  \hline
* & 1 & 0 & -a_{56}/a_{76} & 0 & 0 & * & * \\
* & 0 & 1 & -a_{66}/a_{76}  & 0 & 0 & * & * \\
0 & 0 & 0 & 1 & 0 & \textcolor{red}{a_{76}} & 0 & 0 \\
* & 0 & 0 & -a_{86}/a_{76}  & 1 & 0 & * & * \\
\end{pmatrix}\\
\xrightarrow{(b)}
\begin{pmatrix}
0 & 1 & 0 & a_{56}/a_{76} & 0 & 0 & 0 & 0 \\
0 & 0 & 1 & a_{66}/a_{76} & 0 & 0 & 0 & 0 \\
0 & 0 & 0 & 1 & 0 & 0 & 0 & 0 \\
0 & 0 & 0 & a_{86}/a_{76} & 1& 0 & 0 & 0 \\
  \hline
* & 1 & 0 & 0 & 0 & 0 & * & * \\
* & 0 & 1 & 0  & 0 & 0 & * & * \\
0 & 0 & 0 & 1 & 0 & \textcolor{red}{a_{76}} & 0 & 0 \\
* & 0 & 0 & 0  & 1 & 0 & * & * \\
\end{pmatrix}
&
\xrightarrow{(c)}
\begin{pmatrix}
0 & 1 & 0 & 0 & 0 & 0 & 0 & 0 \\
0 & 0 & 1 & 0 & 0 & 0 & 0 & 0 \\
0 & 0 & 0 & 1 & 0 & 0 & 0 & 0 \\
0 & 0 & 0 & 0 & 1& 0 & 0 & 0 \\
  \hline
* & 1 & 0 & 0 & 0 & 0 & * & * \\
* & 0 & 1 & 0  & 0 & 0 & * & * \\
0 & 0 & 0 & 1 & 0 & \textcolor{red}{a_{76}} & 0 & 0 \\
* & 0 & 0 & 0  & 1 & 0 & * & * \\
\end{pmatrix}.
\end{align*}

If instead the $(4+3, 6)$-entry is the bottom-leftmost smallest order entry, procedures (a), (b) and (c) reduce the matrix into the following form:
\begin{align*}
\begin{pmatrix}
0 & 1 & 0 & 0 & 0 & 0 & 0 & 0 \\
0 & 0 & 1 & 0 & 0 & 0 & 0 & 0 \\
0 & 0 & 0 & 1 & 0 & 0 & 0 & 0 \\
0 & 0 & 0 & 0 & 1& 0 & 0 & 0 \\
  \hline
0 & 1 & 0 & 0 & 0 & * & * & * \\
0 & 0 & 1 & 0  & 0 & * & * & * \\
\textcolor{red}{a_{71}}  & 0 & 0 & 1 & 0 & 0 & 0 & 0 \\
0 & 0 & 0 & 0  & 1 & * & * & * \\
\end{pmatrix}.
\end{align*}

In the general case, notice that the non-zero elements in columns $\alpha$ and $i_{\beta} = j_{b_{\ell}}$ form the following submatrix
\[ \begin{pmatrix} 0& 1\\ * & 1 \end{pmatrix} \quad \text{ or } \quad \begin{pmatrix} 1& 0\\ 1 & * \end{pmatrix} \]
corresponding to the cases $\alpha < i_{\beta}$ and $\alpha> i_{\beta}$ respectively. But throughout the above calculations, the order of the entry $(p+\beta, \alpha)$ is non-negative. Therefore by the calculation for the case $(p,q) = (1,1)$ in Lemma~\ref{11affineclan}, columns $\alpha$ and $i_{\beta}$ are matched, and can be transformed to an affine $(1,1)$-clan matrix in the submatrix $\{\ell , p+\beta\} \times \{\alpha, i_{\beta}\}$.

The example with $a_{76}$ being highlighted in red will be reduced to
\[
\begin{pmatrix}
0 & 1 & 0 & 0 & 0 & 0 & 0 & 0 \\
0 & 0 & 1 & 0 & 0 & 0 & 0 & 0 \\
0 & 0 & 0 & 1 & 0 & 0 & 0 & 0 \\
0 & 0 & 0 & 0 & 1& 0 & 0 & 0 \\
  \hline
* & 1 & 0 & 0 & 0 & 0 & * & * \\
* & 0 & 1 & 0  & 0 & 0 & * & * \\
0 & 0 & 0 & t^{-\ord(a_{76})} & 0 & 1 & 0 & 0 \\
* & 0 & 0 & 0  & 1 & 0 & * & * \\
\end{pmatrix},
\]

while the example with $a_{71}$ being highlighted in red will be reduced to
\begin{align*}
\begin{pmatrix}
0 & 1 & 0 & 0 & 0 & 0 & 0 & 0 \\
0 & 0 & 1 & 0 & 0 & 0 & 0 & 0 \\
0 & 0 & 0 & 1 & 0 & 0 & 0 & 0 \\
0 & 0 & 0 & 0 & 1& 0 & 0 & 0 \\
  \hline
0 & 1 & 0 & 0 & 0 & * & * & * \\
0 & 0 & 1 & 0  & 0 & * & * & * \\
1  & 0 & 0 & t^{-\ord(a_{71})} & 0 & 0 & 0 & 0 \\
0 & 0 & 0 & 0  & 1 & * & * & * \\
\end{pmatrix}.
\end{align*}

Step $7$: Inductively find all the matchings between columns and eliminate extra non-zero entries in the last $q$ rows of the matrix, until every column $i_{a_s} = j_{b_s}$ with $s \in [1,r]$ is matched. The matrix is the affine $(p,q)$-clan matrix we desire.
\end{proof}

\addcontentsline{toc}{section}{Bibliography}
\printbibliography

\end{document}

%% file: preamble.tex
\usepackage{latexsym}
\usepackage{amssymb}
\usepackage{amsthm}
\usepackage{amscd}
\usepackage{amsmath}
\usepackage{mathrsfs}
\usepackage{graphicx}
\usepackage{shuffle}
\usepackage[colorlinks=true, pdfstartview=FitV, linkcolor=blue, citecolor=blue, urlcolor=blue]{hyperref}
\usepackage{mathtools}
\usepackage[bbgreekl]{mathbbol}
\usepackage{mathdots}
\usepackage{ytableau}
\usepackage[enableskew,vcentermath]{youngtab}
\usepackage{tikz-cd}

\usepackage{mathtools}
\usepackage{amssymb}

\usepackage{collcell}
\usepackage{datatool}

\usepackage[colorinlistoftodos]{todonotes}
\usetikzlibrary{chains,scopes}
\usetikzlibrary{shapes.geometric,positioning}

\DeclareSymbolFontAlphabet{\mathbb}{AMSb}
\DeclareSymbolFontAlphabet{\mathbbl}{bbold}

\def\Sp{\mathsf{Sp}}
\def\O{\mathsf{O}}

\def\K{\Bbbk}

\definecolor{darkred}{rgb}{0.7,0,0} 
\newcommand{\defn}[1]{{\color{darkred}\emph{#1}}} 

\numberwithin{equation}{section}
\allowdisplaybreaks[1]

\theoremstyle{definition}
\newtheorem* {theorem*}{Theorem}
\newtheorem* {proposition*}{Proposition}
\newtheorem* {corollary*}{Corollary}
\newtheorem* {conjecture*}{Conjecture}
\newtheorem{theorem}{Theorem}[section]

\theoremstyle{definition}

\newtheorem* {example*}{Example}

\newtheorem{lemma}[theorem]{Lemma}
\theoremstyle{definition}
\newtheorem{definition}[theorem]{Definition}
\theoremstyle{definition}

\newtheorem{proposition}[theorem]{Proposition}
\newtheorem{corollary}[theorem]{Corollary}

\newtheorem{remark}[theorem]{Remark}
\newtheorem*{remark*}{Remark}
\theoremstyle{definition}
\newtheorem {example}[theorem]{Example}
\theoremstyle{definition}

\theoremstyle{definition}

\theoremstyle{definition}

\def\({\left(}
\def\){\right)}

\newcommand{\CC}{\mathbb{C}}

\def\CC{\mathbb{C}}

\def\ZZ{\mathbb{Z}}

\def\GL{\textsf{GL}}

\def\barr{\begin{array}}
\def\earr{\end{array}}
\def\ba{\begin{aligned}}
\def\ea{\end{aligned}}
\def\be{\begin{equation}}
\def\ee{\end{equation}}

\def\quand{\quad\text{and}\quad}

\def\ben{\begin{enumerate}}
\def\een{\end{enumerate}}

\def\bei{\begin{itemize}}
\def\eei{\end{itemize}}

\def\x{\textbf{x}}

\renewcommand{\dim}{\operatorname{dim}}

\def\path{\mathsf{path}}

\def\whSym
{\mathfrak{m}\textsf{Sym}}

\def\whQSym
{\mathfrak{m}\textsf{QSym}}

\def\NN{\mathbb{N}}

\usepackage{listings}
\lstdefinelanguage{Sage}[]{Python}
{morekeywords={False,sage,True},sensitive=true}
\definecolor{dblackcolor}{rgb}{0.0,0.0,0.0}
\definecolor{dbluecolor}{rgb}{0.01,0.02,0.7}
\definecolor{dgreencolor}{rgb}{0.2,0.4,0.0}
\definecolor{dgraycolor}{rgb}{0.30,0.3,0.30}

\usepackage{bbm}

\def\KLt{\Bbbk(\hspace{-0.5mm}(t)\hspace{-0.5mm})}
\def\KPt{\Bbbk[\hspace{-0.5mm}[t]\hspace{-0.5mm}]}

\newcommand{\ord}{\operatorname{ord}}